\documentclass[final,3p,mathptmx]{elsarticle}

\usepackage{amsmath,amsthm,amssymb,amsfonts}
\usepackage{color}
\usepackage{rotating}

\newcommand{\supp}{\mathop{\mathrm{supp}}}

\newcommand{\diag}{\mathop{\mathrm{diag}}}

\theoremstyle{plain}
\newtheorem{thm}{Theorem}[section]

\newtheorem{prop}[thm]{Proposition}

\theoremstyle{definition}
\newtheorem{defn}{Definition}[section]
\newtheorem{exmp}{Example}[section]

\theoremstyle{remark}
\newtheorem{rem}[thm]{Remark}

\everymath{\displaystyle}
\begin{document}

\title{
Optimal H\"older-Zygmund exponent of semi-regular refinable functions
}

\author[univie]{Maria Charina}
\ead{maria.charina@univie.ac.at}

\author[unifi]{Costanza Conti}
\ead{costanza.conti@unifi.it}

\author[unimib]{Lucia Romani}
\ead{lucia.romani@unimib.it}

\author[TU]{Joachim St\"ockler}
\ead{joachim.stoeckler@math.tu-dortmund.de}

\author[unimib]{Alberto Viscardi}
\ead{alberto.viscardi@unimib.it}

\address[univie]{Fakult\"at f\"ur Mathematik, Universit\"at Wien, Oskar-Morgenstern-Platz 1, 1090 Wien, Austria}
\address[unifi]{Dipartimento di Ingegneria Industriale, Universit\`a degli Studi di Firenze, viale Morgagni 40/44, 50134 Firenze, Italy}
\address[unimib]{Dipartimento di Matematica e Applicazioni, Universit\`a degli Studi di Milano-Bicocca, via Roberto Cozzi 55, 20126 Milano, Italy}
\address[TU]{Institut f\"ur Angewandte Mathematik, TU Dortmund, Vogelpothsweg 87, D-44227 Dortmund, Germany}

\date{}

\begin{abstract}
The regularity of refinable functions has been investigated deeply in the past 25 years using Fourier analysis, wavelet analysis, 
restricted and joint spectral radii techniques. However the shift-invariance of the underlying regular setting is crucial for these
approaches. We propose an efficient method based on wavelet tight frame decomposition techniques for estimating 
H\"older-Zygmund regularity of univariate semi-regular refinable functions generated, e.g., by subdivision schemes defined on semi-regular meshes
$\mathbf{t}\;=\;-h_\ell\mathbb{N}\cup\{0\}\cup h_r\mathbb{N}$, $h_\ell,h_r \in (0,\infty)$.
To ensure the optimality of this method, we provide a new characterization of H\"older-Zygmund spaces based on suitable irregular 
wavelet tight frames. Furthermore, we present proper tools for computing the corresponding frame coefficients in the semi-regular 
setting. We also propose a new numerical approach for estimating the optimal H\"older-Zygmund exponent of refinable functions which is more
efficient than the linear regression method. We illustrate our results with
several examples of known and new semi-regular subdivision schemes with a potential use in blending curve design. 

\begin{keyword}
wavelet tight frames \sep semi-regular refinement \sep Dubuc-Deslauriers frames \sep H\"older-Zygmund regularity
\end{keyword}

\bigskip
\noindent{\bf Classification (MSCS):  42C40, 42C15, 65D17}

\end{abstract}

\maketitle

\section{Introduction and notation} \label{sec:intro}

This paper presents a fast and efficient method for computing the optimal (critical) H\"older-Zygmund regularity of a certain class of non-shift-invariant univariate 
refinable functions, the so-called semi-regular refinable functions generated e.g. by binary subdivision
\cite{MR1687779,MR1356994} defined on the meshes
\begin{equation}\label{eq:mesh}
  {\mathbf t}\;=\;-h_\ell\mathbb{N}\cup\{0\}\cup h_r\mathbb{N},\quad h_\ell,h_r \in (0,\infty).
\end{equation}
It is well known that a family $\{\phi_{k}\,:\,k\in\mathbb{Z}\}$ of refinable functions $\phi_{k} \in L_2(\mathbb{R})$ assembled in 
a bi-infinite column vector $\Phi=[\phi_{k}\,:\,k\in\mathbb{Z}]$ satisfies the refinement equation
\begin{equation} \label{eq:ref}
   \Phi = \mathbf{P}\, \Phi(2 \cdot)
\end{equation}
with a real-valued bi-infinite matrix $\mathbf{P}$. In the semi-regular case, finitely many (corresponding to a certain 
neighborhood of the origin) of the elements in $\Phi$ can not be expressed as integer shifts of any other function in 
$\Phi$ and \eqref{eq:ref} reduces to finitely many different scalar-valued refinement equations.
The assumption \eqref{eq:mesh} on the mesh becomes vital only in Section \ref{sec:num}.

Our method relies on a new characterization of H\"older-Zygmund spaces.  It generalizes successful wavelet frame methods 
\cite{MR3495345, MR1150048, MR2683008, MR1162107, MR1083586, MR1228209, MR2948962, MR1289147} from the regular to the semi-regular and even to the irregular setting and is the first step towards a better understanding of regularity at extraordinary vertices 
\cite{MR2415757, Warren:2001:SMG:580358} in the bivariate case. In comparison to the method in \cite{MR1687779},  our approach yields numerical estimates for the optimal H\"older-Zygmund regularity of a refinable function without requiring any ad hoc regularity estimates for the corresponding subdivision scheme. Our numerical estimates turn out to be optimal in all considered cases and require fewer computational steps than the standard linear regression method.

In the regular case, the wavelet frame methods rely on the characterization of Besov spaces
$B^{r}_{p,q}(\mathbb{R})$ provided by Lemari\'e and Meyer \cite{MR864650} in their follow-up on the results by Frazier and Jawerth \cite{MR808825}.

\begin{thm}[\cite{MR1228209}, Section 6.10] \label{thm:wave_char}
Let $s>0$ and $1\leq p,q\leq\infty$.
Assume
$$
 \big\{\phi_k=\phi_0(\cdot-k) \ : \ k\in\mathbb{Z} \big\} \cup \big\{ \psi_{j,k}=2^{(j-1)/2}\psi_{1,0}(2^{j-1}\cdot-k) \ : \ k\in\mathbb{Z}, j\in\mathbb{N}\big\} \subset \mathcal{C}^s(\mathbb{R})
$$
is a compactly supported orthogonal wavelet system with $v$ vanishing moments.  Then, for $r \in (0, \min(s,v))$,
\[
 B^{r}_{p,q}(\mathbb{R}) \;=\; \left\{\; \sum_{k\in\mathbb{Z}}a_k\phi_k+\sum_{j\in\mathbb{N}}\sum_{k\in\mathbb{Z}}b_{j,k}\psi_{j,k}\;: \; \{a_k\}_{k\in\mathbb{Z}} \in \ell_p(\mathbb{Z}), \ \left\{2^{j\left(r+\frac{1}{2}-\frac{1}{p}\right)}\left\|\{b_{j,k}\}_{k\in\mathbb{Z}}\right\|_{\ell_p} \;\right\}_{j\in\mathbb{N}} \in \ell_q(\mathbb{Z}) \right\}.
\]
\end{thm}

To be able to apply Theorem \ref{thm:wave_char}, i.e. to extract the regularity of a given function $f$ from the decay of its coefficients $\{a_{k}=\langle f,\phi_{k} \rangle\ \,:\, k\in\mathbb{Z}\}$ and $\{b_{j,k}=\langle f,\psi_{j,k}\rangle\,:\,j\in\mathbb{N},\ k\in\mathbb{Z}\}$, one must first compute these inner products. In the context of subdivision, the analytic expressions neither of the analyzed function $f$ nor of the refinable functions $\phi_k$
are usually known. However, in the regular (shift-invariant) setting, the desired inner products can be computed explicitly (or numerically) using results of \cite{MR1302255}.  In the general  non-shift-invariant case, the task becomes
overwhelming and is far from being understood.  In the semi-regular case, however,  both the suitable wavelet
tight frames exist, e.g. the ones generated by B-splines \cite{MR2082157,MR2110512} or by Dubuc-Deslauriers refinable
functions \cite{first_paper}, and, similarly to \cite{MR2692268,first_paper}, the corresponding frame coefficients
can be computed.

However, Theorem \ref{thm:wave_char} does not cover the case of semi-regular wavelet tight frames, since we lose the orthogonality and, most importantly, the shift-invariance. This paper provides a generalization of Theorem \ref{thm:wave_char} for function systems
\begin{equation} \label{def:F}
 \mathcal{F}=\big\{\phi_k \, : \, k\in\mathbb{Z} \big\} \cup \big\{ \psi_{j,k} \, : \, j\in\mathbb{N}, \ k\in\mathbb{Z}\big\}
\end{equation}
with the following properties
\begin{itemize}
\item[\color{black}(I)] $\mathcal{F}$ forms a (Parseval/normalized) tight frame for $L^2(\mathbb{R})$, i.e.
	\begin{equation}\label{eq:TF}
     f\;=\;\sum\limits_{k\in\mathbb{Z}}\;\langle f,\phi_k\rangle \phi_k\;+\;\sum\limits_{j\in\mathbb{N}}\;\sum\limits_{k\in\mathbb{Z}}\;\langle f,\psi_{j,k}\rangle\;\psi_{j,k}, \quad f\in L^2(\mathbb{R});
  \end{equation}
\item[\color{black}(II)] there exists a constant $C_{supp}>0$ such that
		\begin{equation}\label{eq:bound_supp} \sup_{k\in\mathbb{Z}}\{|\supp(\phi_k)|\}\;\leq\;C_{supp}\quad\textrm{and}       \quad\sup_{k\in\mathbb{Z}}\{|\supp(\psi_{j,k})|\}\;\leq\;C_{supp}\;2^{-j},\quad  j\in\mathbb{N};
		\end{equation}
\item[\color{black}(III)] there exists a constant $C_\Gamma>0$ such that for every bounded interval $K \subset \mathbb{R}$  the sets
		\begin{eqnarray*}
		 \Gamma_0(K)\;=\;\{\;k\in\mathbb{Z}\;:\;\supp(\phi_k)\cap K\;\not=\;\emptyset\;\}
		\quad \hbox{and} \quad
			\Gamma_j(K)\;=\;\{\;k\in\mathbb{Z}\;:\;\supp(\psi_{j,k})\cap K\;\not=\;\emptyset\;\},\quad j\in\mathbb{N},
		\end{eqnarray*}
		satisfy
		\begin{equation}\label{eq:bound_Gam}
		   |\Gamma_j(K)|\;\leq\;C_\Gamma(2^j|K|+1),\quad  j\geq 0;
		\end{equation}					
\item[\color{black}(IV)] $\mathcal{F}$ has $v\in\mathbb{N}$ vanishing moments, i.e.
	\begin{equation}\label{eq:vm}
	   \int_\mathbb{R}\;x^n\;\psi_{j,k}(x)\;dx\;=\; 0,\quad n\in\{0,\dots,v-1\}, \quad  j\in\mathbb{N},
		\quad k \in\mathbb{Z},
	\end{equation}
 and there exists a sequence of points $\{x_{j,k}\,:\, j\in\mathbb{N},\ k\in\mathbb{Z}\}$ such that, 
for every $0\leq r\leq v$, there exists a constant $C_{vm,r}>0$ such that
		\begin{equation}\label{eq:bound_vm}
		   \sup_{k\in\mathbb{Z}}\;\int_{\mathbb{R}}|x|^r|\psi_{j,k}(x+x_{j,k})|dx\;\leq\;C_{vm,r}\;
			 2^{-j\left(r+\frac{1}{   2}\right)};
			\end{equation}		
\item[\color{black}(V)] $\mathcal{F} \subset \mathcal{C}^s(\mathbb{R})$, $s>0$, and for every $0 \leq r\leq s$ there exists a constant $C_{sm,r}>0$ such that
		\begin{equation}\label{eq:bound_norm}
		 \sup_{k\in\mathbb{Z}}\{\|\phi_k\|_{\mathcal{C}^r}\}\;\leq\;C_{sm,r}\quad\textrm{and}\quad\sup_{k\in\mathbb{Z}}     \{\|\psi_{j,k}\|_{\mathcal{C}^r}\}\;\leq\;C_{sm,r}\;2^{j\left(r+\frac{1}{2}\right)},\quad j\in\mathbb{N}.
		\end{equation}
\end{itemize}

Estimate \eqref{eq:bound_vm} expresses localization condition for the framelets $\psi_{j,k}$. Note that \eqref{eq:bound_vm} is  implied by conditions (II) and (V) for $0\leq r\leq s$. Indeed, choosing $x_{j,k}$ to be the midpoint of $\supp(\psi_{j,k})$, $j\in\mathbb{N}$, $k\in\mathbb{Z}$, we get
\[
\int_\mathbb{R}\;|x|^r\;|\psi_{j,k}(x+x_{j,k})|dx \;\leq\; C_{sm,0}\;\left(\frac{|\supp(\psi_{j,k})|}{2}\right)^r\;2^{j/2}\;|\supp(\psi_{j,k})|\;\leq\;\frac{C_{supp}^{r+1}\;C_{sm,0}}{2^r}\;2^{-j\left(r+\frac{1}{2}\right)}.
\]
We state the assumptions (II) and (V) separately to emphasize their duality, which becomes even more
evident in the statements of Propositions \ref{prop:Holder_suff} and \ref{prop:Holder_nec} in Section
\ref{sec:Hol}. Indeed, Theorems \ref{thm:wave_char} and \ref{thm:main} require $0<r<\min(s,v)$, where
the value of $s$ or $v$ affects only one of the inclusions in either Proposition \ref{prop:Holder_suff}
or in Proposition \ref{prop:Holder_nec}. Moreover, stating (IV) and (V) separately, we can easily generalize our results to the case 
of dual frames with the analysis frame satisfying (II) and (IV) and
the synthesis frame satisfying (III) and (V).  In the regular case, a natural choice in \eqref{eq:bound_vm} is $x_{j,k}=2^{-j} k$.  In general, even if the system $\mathcal{F}$ is non-shift-invariant, the points $x_{j,k}$
ensure the quasi-uniform (standard concept in the context of spline and finite element methods) behavior of the framelets over $\mathbb{R}$  and act as the center for every element
$\psi_{j,k}$. The other assumptions also manifest the quasi-uniform behavior of the framelets over $\mathbb{R}$.

The setting described by assumptions (I)-(V) includes some cases not addressed in the results of Frazier and Jawerth \cite{MR808825} or of Cordero and 
Gr\"ochenig in \cite{MR2067914}. The results of \cite{MR808825} require that the elements of $\mathcal{F}$ in the decomposition of $B^{r}_{p,q}(\mathbb{R})$ are linked to dyadic intervals. The results in \cite{MR2067914} impose the so-called localization property which implies that the system $\mathcal{F}$ is semi-orthogonal (in particular, non-redundant).

On the other hand, one could view the quite natural and application oriented assumptions (I)-(V) to be somewhat restrictive, since these assumptions were designed to fit wavelet tight frames $\mathcal{F}$ constructed using results of \cite{MR2110512}. For such function families $\mathcal{F}$, there exists a sequence of bi-infinite matrices $\{\mathbf{Q}_j \,:\,j\in\mathbb{N}\}$ such that the column vectors
$$
 \Psi_j=[\psi_{j,k}\,:\,k\in\mathbb{Z}], \quad j \in \mathbb{N}, \quad \hbox{and} \quad
 \Phi=[\phi_{k}\,:\,k\in\mathbb{Z}]
$$
satisfy
\begin{equation} \label{eq:WTF}
 \Psi_j\;=\;2^{j/2}\;\mathbf{Q}_j^T \;\Phi(2^{j}\cdot),\quad j\in\mathbb{N}.
\end{equation}
Such bi-infinite matrices $\{\mathbf{Q}_j \,:\,j\in\mathbb{N}\}$ are e.g the ones constructed in
\cite{first_paper} for the family of Dubuc-Deslauriers subdivision schemes \cite{ MR2805717, MR2855428, MR982724}.
For function families $\mathcal{F}$ satisfying \eqref{def:F}, assumptions (II)-(V) reflect the properties of the matrices $\{\mathbf{Q}_j \,:\,j\in\mathbb{N}\}$: (II) controls the support of the columns of the $\mathbf{Q}_j$s, (III) controls the slantedness of the $\mathbf{Q}_j$s and (IV) and (V) are linked to eigenproperties of the $\mathbf{Q}_j$s.

Nevertheless, the spirit of assumptions (I)-(V) merges with the spirit of atoms and molecules in \cite{MR808825} and compactly supported orthogonal wavelet systems, for which (I)-(V) are also satisfied. These similarities are also visible in the structure of the proofs of Propositions \ref{prop:Holder_suff} and \ref{prop:Holder_nec}.

For the sake of completeness, we point out that our setting includes some of the wavelet frames considered in \cite{MR2534255} for which a characterization of the spaces $B^r_{2,2}(\mathbf{R})$, $r\in\mathbb{R}$, is given. However, those frames are shift-invariant, i.e \eqref{def:F} holds with block $2$-slanted $\{\mathbf{Q}_j \,:\,j\in\mathbb{N}\}$. The approach in \cite{MR2534255} applies Fourier techniques that are not feasible in our case, due to the lack of shift-invariance. The lack of shift-invariance makes also the techniques in \cite{MR2094588} inapplicable in our case. In \cite{MR2094588}, the authors characterize Lebesgue and Sobolev spaces via shift-invariant wavelet tight frames.

We concentrate on the case $p=q=\infty$, the one most relevant for subdivision. Our main result,
Theorem \ref{thm:main} whose proof is given in Sections \ref{sec:Hol} and \ref{sec:Zyg}, reads as follows.

\begin{thm} \label{thm:main}
	Let $s>0$ and $v\in\mathbb{N}$. Assume $\mathcal{F} \subset \mathcal{C}^s(\mathbb{R})$
	satisfies assumptions (I)-(V) with $v$ vanishing moments.
	Then, for $r \in (0,\min(s,v))$,
	\[
 B^{r}_{\infty,\infty}(\mathbb{R}) \;=\; \left\{\; \sum_{k\in\mathbb{Z}}a_k\phi_k+\sum_{j\in\mathbb{N}}\sum_{k\in\mathbb{Z}}b_{j,k}\psi_{j,k}\;: \; \{a_k\}_{k\in\mathbb{Z}} \in \ell_\infty(\mathbb{Z}), \ \left\{2^{j\left(r+\frac{1}{2}\right)}\left\|\{b_{j,k}\}_{k\in\mathbb{Z}}\right\|_{\ell_\infty} \;\right\}_{j\in\mathbb{N}} \in \ell_\infty(\mathbb{Z}) \right\}.
\]
\end{thm}

The paper is organized as follows. In Subsection \ref{subsec:spaces}, we define the function and sequence spaces that we consider. In Section \ref{sec:Hol}, Theorem \ref{thm:Holder_gen} gives the proof of Theorem \ref{thm:main} in the case $r\in(0,\infty)\setminus\mathbb{N}$ and, in Section \ref{sec:Zyg}, Theorem \ref{thm:Zygmund} provides the proof for $r \in\mathbb{N}$. 
We would like to emphasize that the results in Sections \ref{sec:Hol} and \ref{sec:Zyg} are true in regular, semi-regular and irregular
cases. Theorem \ref{thm:main} implies the norm equivalence between Besov spaces $B_{\infty,\infty}^r(\mathbb{R})$ and the 
sequence spaces $\ell_{\infty, \infty}^r$, $r \in (0, \infty)$, see Remark \ref{rem:norm_eq}. 
The proofs in Sections \ref{sec:Hol} and \ref{sec:Zyg} are reminiscent of the continuous wavelet transform techniques in 
\cite{MR1162107, MR1228209} and references therein. In Section \ref{sec:num},  we illustrate our results with
several examples. There the structure of the 
mesh in \eqref{eq:mesh} becomes important. In particular, we use  wavelet tight frames constructed in \cite{first_paper}, to approximate the H\"older-Zygmund regularity of semi-regular subdivision schemes based on B-splines, the family of Dubuc-Deslauriers subdivision schemes and interpolatory radial basis functions (RBFs) based subdivision. Semi-regular B-spline and Dubuc-Deslauriers schemes were introduced, e.g in \cite{MR1687779,MR1356994, Warren:2001:SMG:580358}. The construction of semi-regular RBFs based schemes is
our generalization of \cite{MR2231695,MR2578850} to the semi-regular case. We would like to point out that such semi-regular schemes 
can be used for blending curve pieces with different properties.  

\subsection{Function and sequence spaces: notation} \label{subsec:spaces}

We use the standard notation for the  function spaces $\mathcal{C}^s(\mathbb{R})$, $s \in \mathbb{N}_0$, the H\"older spaces
\[
 \mathcal{C}^s(\mathbb{R})\;=\;\left\{\;f\in\mathcal{C}^\ell(\mathbb{R})\;:\;\sup_{x,h\in\mathbb{R}} \;\frac{|f^{(\ell)}(x+h)-f^{(\ell)}(x)|}{|h|^\alpha}\;<\;\infty\;\right\}, \quad s=\ell+\alpha, \quad
\ell \in \mathbb{N}_0, \quad \alpha \in (0,1),
\]
with $f^{(\ell)}$ denoting the $\ell$-th derivative of $f$, the \emph{Zygmund class}
\begin{equation} \label{eq:Zyg}
 \Lambda(\mathbb{R})\;=\;\left\{\;f:\mathbb{R}\rightarrow\mathbb{R}\;:\;\sup_{x,h\in\mathbb{R}}\frac{|f(x+h)-2f(x)+f(x-h)|}{|h|}\;<\;\infty\;\right\},
\end{equation}
the Lebesque spaces
$L^p(\mathbb{R})$, $1 \le p < \infty$, and for sequence spaces $\ell_p(\mathbb{Z})$, $1 \le p \le \infty$.

Besov spaces $B^r_{p,q}(\mathbb{R})$, e.g in \cite{MR1228209}, are defined by
\begin{equation}\label{eq:Besov}
B^r_{p,q}(\mathbb{R})\;:=\;\left\{\;f\in L^p(\mathbb{R})\;:\;\|f\|_{B^r_{p,q}}=\left\|\{2^{j r}\omega^{[r]+1}_p(f,2^{-j})\}_{j\in\mathbb{N}}\right\|_{\ell^q}<\infty\;\right\},\quad 1\leq p,q\leq\infty,\quad r \in (0,\infty)
\end{equation}
with the \emph{$p$-th modulus of continuity of order $n\in\mathbb{N}$}
\[\omega^n_p(f,x)\;=\;\sup_{|h|\leq x}\|\Delta^{n}_h(f,\cdot)\|_{L^p}\]
and the \emph{difference operator of order $n \in \mathbb{N}$ and step $h>0$}
\[\Delta^n_h(f,x)\;=\;\sum_{\ell=0}^{n} {{n}\choose{\ell}} (-1)^{n-\ell}f(x+\ell h).\]
The special case $p=q=\infty$ reduces to
\begin{equation}\label{eq:Hol-Zyg}
 B^r_{\infty,\infty}(\mathbb{R})\;=\;\left\{\begin{array}{cl}
	\mathcal{C}^r(\mathbb{R})\cap L^\infty(\mathbb{R}), & \textrm{if } r\in(0,\infty)\setminus\mathbb{N}, \\ \\
	\{\;f\in \mathcal{C}^{r-1}(\mathbb{R})\cap L^\infty(\mathbb{R})\;:\;f^{(r-1)}\in\Lambda(\mathbb{R})\;\}, & \textrm{if } r\in\mathbb{N}.
\end{array}\right.
\end{equation}
The corresponding sequence spaces $\ell_{p,q}^r$, $r \in (0,\infty)$, are defined, for $1 \le p \le \infty$ and  $1 \le q <\infty$,  by
$$
 \ell_{p,q}^r=\Big\{ (a,b) \in \mathbb{Z} \times (\mathbb{N} \times \mathbb{Z}) \,:\, \|(a,b)\|_{\ell_{p,q}^r}=
 \Big(\|a\|_{\ell_p}^q  + \sum_{j=1}^\infty 2^{j(r+\frac{1}{2}-\frac{1}{p})q} \|\{b_{j,k}\}_{k\in\mathbb{Z}}\|_{\ell_p}^q\Big)^{1/q} \Big\}
$$
and, for $1 \le p \le \infty$ and $q=\infty$, by
$$
 \ell_{p,\infty}^r=\Big\{ (a,b) \in \mathbb{Z} \times (\mathbb{N} \times \mathbb{Z}) \,:\, \|(a,b)\|_{\ell_{p,\infty}^r}=
 \max\big\{ \|a\|_{\ell_p},  \sup_{j \in \mathbb{N}} 2^{j(r+\frac{1}{2}-\frac{1}{p})} \|\{b_{j,k}\}_{k\in\mathbb{Z}}\|_{\ell_p} \big\} \Big\} \,.
$$

\section{Characterization of H\"older spaces $B^r_{\infty,\infty}(\mathbb{R})$, $r\in(0,\infty)\setminus\mathbb{N}$} \label{sec:Hol}

In this section, in Theorem~\ref{thm:Holder_gen} we characterize the H\"older spaces
$B^r_{\infty,\infty}(\mathbb{R})=\mathcal{C}^r(\mathbb{R})\cap L^\infty(\mathbb{R})$ for
$r\in(0,\infty)\setminus\mathbb{N}$ in terms of the function system $\mathcal{F}$ in
\eqref{def:F}. The proof of Theorem~\ref{thm:Holder_gen} follows after
Propositions \ref{prop:Holder_suff} and \ref{prop:Holder_nec} that stress the duality between conditions (IV) and (V). Proposition \ref{prop:Holder_suff}, provides the inclusion ''$\supseteq$``
under assumptions (III), (V) and $r \in(0,\min(s,1))$. Whereas Proposition \ref{prop:Holder_nec}
yields the other inclusion ''$\subseteq$`` under assumptions
(I), (II), (IV) and $r \in (0,1)$. The proof of Theorem \ref{thm:Holder_gen} then extends the argument of Propositions \ref{prop:Holder_suff} and \ref{prop:Holder_nec} to the case $r>1$, $r\not\in\mathbb{N}$.  Our results show that the continuous wavelet transform techniques from
\cite{MR1162107, MR1228209} and references therein are almost directly applicable in the irregular setting.

\begin{thm} \label{thm:Holder_gen} Let $s>0$ and $v\in\mathbb{N}$. Assume $\mathcal{F}$ satisfies (I)-(V) with $v$ vanishing moments. Then, for $r \in(0,\min(s,v))\setminus\mathbb{N}$,
\begin{equation*}
 B^{r}_{\infty,\infty}(\mathbb{R}) \;=\; \left\{\; \sum_{k\in\mathbb{Z}}a_k\phi_k+\sum_{j\in\mathbb{N}}\sum_{k\in\mathbb{Z}}b_{j,k}\psi_{j,k}\;: \; 
 (a,b) \in \ell_{\infty,\infty}^r \ \hbox{with} \ \ a=\{a_k\}_{k \in \mathbb{Z}}, \ b=\{b_{j,k}\}_{j\in\mathbb{N}, k \in \mathbb{Z}} \right\}.
\end{equation*}
\end{thm}	

We start by proving the following result.

\begin{prop} \label{prop:Holder_suff}
	Let $s>0$. Assume $\mathcal{F}$ satisfies (III) and (V). Then, for $r \in (0,\min(s,1))$,
\[
 B^{r}_{\infty,\infty}(\mathbb{R}) \;\supseteq\; \left\{\; \sum_{k\in\mathbb{Z}}a_k\phi_k+\sum_{j\in\mathbb{N}}\sum_{k\in\mathbb{Z}}b_{j,k}\psi_{j,k}\;: \; (a,b) \in \ell_{\infty,\infty}^r \ \hbox{with} \ \ a=\{a_k\}_{k \in \mathbb{Z}}, \ b=\{b_{j,k}\}_{j\in\mathbb{N}, k \in \mathbb{Z}} \right\}.
\]
\end{prop}

\begin{proof}
We consider $f(x)\;=\;f_0(x)\;+\;g(x)$, $x\in\mathbb{R}$, where
\begin{equation}\label{def:f0_g}
 f_0(x)\;=\;\sum\limits_{k\in\mathbb{Z}}\;a_{k}\;\phi_k(x)\quad\textrm{and}\quad g(x)\;=\;\sum\limits_{j\in\mathbb{N}}\;\sum\limits_{k\in\mathbb{Z}}\;b_{j,k}\;\psi_{j,k}(x),
\end{equation}
with finite
\begin{equation}\label{eq:coeff_bound}
 C_a:=\sup_{k\in\mathbb{Z}}\{|a_k|\} \quad \textrm{and} \quad 
 C_b:=\sup_{j \in \mathbb{N}} 2^{j\left(r+\frac{1}{2}\right)} \sup_{k\in\mathbb{Z}} \{|b_{j,k}|\}.
\end{equation}
Since on every open bounded interval in $\mathbb{R}$ the sum defining $f_0$ is finite due to (III), 
we have $f_0\in \mathcal{C}^s(\mathbb{R})\subseteq \mathcal{C}^r(\mathbb{R})$ due to $r<s$. Moreover, by (III) and (V), we obtain
\begin{equation}\label{es:f0}
 \|f_0\|_{L^\infty}\;\leq\;C_a\;C_\Gamma\;C_{sm,0}\;<\;\infty.
\end{equation}
Analogously, since $r>0$, we have
\begin{equation} \label{es:g}
\|g\|_{L^\infty}\;\leq\;C_b\;C_\Gamma\;C_{sm,0}\;\sum\limits_{j\in\mathbb{N}}\;2^{-jr}\;<\;\infty,
\end{equation}
thus, $f\in L^\infty(\mathbb{R})$.
Let $x,h\in\mathbb{R}$. By \eqref{eq:coeff_bound}, we get
\[\begin{array}{rcl}
\left|\;g(x+h)\;-\;g(x)\;\right| &\leq& \sum\limits_{j\in\mathbb{N}}\;\sum\limits_{k\in\mathbb{Z}}\;\left|\;b_{j,k}\;\right|\;\left|\;\psi_{j,k}(x+h)\;-\;\psi_{j,k}(x)\;\right| \\ \\
&\le& C_b\,|h|^r\;\sum\limits_{j\in\mathbb{N}}\;\frac{2^{-j\left(r+\frac{1}{2}\right)}}{|h|^{r}}\;\sum\limits_{k\in\mathbb{Z}}\;\left|\;\psi_{j,k}(x+h)\;-\;\psi_{j,k}(x)\;\right|\;. \\ \\
\end{array}
\]
Since there exists $J\in\mathbb{Z}$ such that
\[2^{-J}\;<\;|h|\;\leq\;2^{-J+1}\;,\]
we have
\[\begin{array}{rcl}
\left|\;g(x+h)\;-\;g(x)\;\right| &\leq&  C_b\,|h|^r\;\sum\limits_{j\in\mathbb{N}}\;2^{(J-j)r-j/2}\;\sum\limits_{k\in\mathbb{Z}}\;\left|\;\psi_{j,k}(x+h)\;-\;\psi_{j,k}(x)\;\right| \\ \\
&=& C_b\,|h|^r\;\left(\;A\;+\;B\;\right)\;,
\end{array}\]
where
{
\begin{equation} \label{eq:AB1} \begin{array}{c}
A\;=\;\sum\limits_{j=1}^{J-1}\;2^{(J-j)r-j/2}\;\sum\limits_{k\in\mathbb{Z}}\;\left|\;\psi_{j,k}(x+h)\;-\;\psi_{j,k}(x)\;\right|\quad\textrm{and}\\ \\
B\;=\;\sum\limits_{j=J}^{\infty}\;2^{(J-j)r-j/2}\;\sum\limits_{k\in\mathbb{Z}}\;\left|\;\psi_{j,k}(x+h)\;-\;\psi_{j,k}(x)\;\right|.
\end{array}\end{equation}
}
If $J\leq 1$, $A\;=\;0$. Otherwise, for every $\epsilon>0$ with $r<r+\epsilon<\min(s,1)$, due to (V) we have
\[
\left|\;\psi_{j,k}(x+h)\;-\;\psi_{j,k}(x)\;\right|\;\leq\;C_{sm,r+\epsilon}\;2^{j\left(r+\epsilon+\frac{1}{2}\right)}\;|h|^{r+\epsilon}\;\leq\;C_{sm,r+\epsilon}\;2^{(j-J)(r+\epsilon)}\;2^{j/2}\;2^{r+\epsilon}\;,
\]
and, by (III), the sum in $A$ over $k$ has at most
\[
|\Gamma_j(x+h)|\;+\;|\Gamma_j(x)|\;\leq\;2\;C_\Gamma
\]
non-zero elements. Thus,
\begin{equation}\label{eq:A_bound}
A\;\leq\;2^{r+\epsilon+1}\;C_\Gamma\;C_{sm,r+\epsilon}\;\sum\limits_{j=1}^{J-1}\;(2^{-\epsilon})^{(J-j)}\;\leq
\;4\;C_\Gamma\;C_{sm,r+\epsilon}\;\sum\limits_{j=1}^{J-1}\;(2^{-\epsilon})^{j}.
\end{equation}
Therefore, since $\epsilon>0$, $A$ is bounded. To conclude the proof, we observe that, by (V),
\[
B\;\leq\;2\;C_\Gamma\;C_{sm,0}\;\sum\limits_{j=J}^{\infty}\;(2^{-r})^{j-J}\;=\;2\;C_\Gamma\;C_{sm,0}\;\frac{1}{1-2^{-r}}.
\]
Thus $|g(x+h)-g(x)|/|h|^r$ is uniformly bounded in $x$ and $h$, which leads to $g\in B^r_{\infty,\infty}(\mathbb{R})$ and $f\in B^r_{\infty,\infty}(\mathbb{R})$ with $\|f\|_{B_{\infty,\infty}^r} \le C \,\|(a,b)\|_{\ell_{\infty,\infty}^r}$ for some constant $C>0$.
\end{proof}

Next, we give a proof of Proposition \ref{prop:Holder_nec}.

\begin{prop} \label{prop:Holder_nec}
Assume $\mathcal{F} \subset \mathcal{C}^0(\mathbb{R})$ with uniformly bounded
$\{\phi_k\,:\,k\in\mathbb{Z}\}$ satisfies (I), (II) and (IV) with $1$ vanishing moment.
	Then, for $r\in (0,1)$,
\[
 B^{r}_{\infty,\infty}(\mathbb{R}) \;\subseteq\; \left\{\; \sum_{k\in\mathbb{Z}}a_k\phi_k+\sum_{j\in\mathbb{N}}\sum_{k\in\mathbb{Z}}b_{j,k}\psi_{j,k}\;: \; (a,b) \in \ell_{\infty,\infty}^r \ \hbox{with} \ \ a=\{a_k\}_{k \in \mathbb{Z}}, \ b=\{b_{j,k}\}_{j\in\mathbb{N}, k \in \mathbb{Z}} \right\}.
\]
\end{prop}

\begin{proof}
Consider $f\in B^r_{\infty,\infty}(\mathbb{R})\cap L^2(\mathbb{R})$. We choose a representative of $f$ in \eqref{eq:TF} with 
coefficients $a_k=\langle f,\phi_k\rangle$ and $b_{j,k}=\langle f,\psi_{j,k}\rangle$.
On one hand, due to (II) and the uniform boundedness of $\Phi$, there exists $C_\phi>0$  such that
\begin{equation}\label{eq:coarse_bound}
|a_k|\;\leq\;C_\phi\;\|f\|_\infty,\quad k\in\mathbb{Z}.
\end{equation}
On the other hand, with $x_{j,k}$ as in (IV), we can exploit the vanishing moment of the tight frame and the regularity of $f$ to get
\begin{equation} \label{eq:vm_trick} \begin{array}{rcl}
\left|b_{j,k}\right| &=& \left|\;\int_{\mathbb{R}}\;f(x)\;\psi_{j,k}(x)\;dx\;\right| \;=\; \left|\;\int_{\mathbb{R}}\;\left(\;f(x)\;-\;f(x_{j,k})\;\right)\;\psi_{j,k}(x)\;dx\;\right| \\ \\
&\leq& \|f\|_{B^r_{\infty,\infty}} \int_{\mathbb{R}}\;|x-x_{j,k}|^r\;\left|\psi_{j,k}(x)\right|\;dx \\ \\
&=& \|f\|_{B^r_{\infty,\infty}} \int_{\mathbb{R}}\;|x|^r\;\left|\psi_{j,k}(x+x_{j,k})\right|\;dx\;\leq\; 2^{-j\left(r+\frac{1}{2}\right)}\;C_{vm,r}\;\|f\|_{B^r_{\infty,\infty}}
\end{array}\end{equation}
For a general $f\in B^r_{\infty,\infty}(\mathbb{R})$ the claim follows by a density argument. Thus, there exists a constant $C>0$ such that
$\|f\|_{B_{\infty, \infty}^r} \ge C \|(a,b)\|_{\ell_{\infty,\infty}^r}$.
\end{proof}

\begin{rem} \label{rem:only_vm}
In Proposition \ref{prop:Holder_nec}, there is no need for the tight frame to be more than continuous - only the vanishing moment matters. The same phenomenon happens for the inclusion $\subseteq$ in Theorem \ref{thm:Holder_gen} - the number of vanishing moments being the key ingredient
for its proof. On the other hand, the regularity of the wavelet tight frame $\mathcal{F}$
plays the key role both in Proposition \ref{prop:Holder_suff} and in the proof of the inclusion $\supseteq$ in Theorem \ref{thm:Holder_gen}. This explains the duality between assumptions (IV) and (V).
\end{rem}

We are now ready to complete the proof of Theorem \ref{thm:Holder_gen}.

\begin{proof}[Proof of Theorem \ref{thm:Holder_gen}]
For the case $r\in(0,1)$ see Propositions \ref{prop:Holder_suff} and \ref{prop:Holder_nec}. Let $r=n+\alpha$, with $n\in\mathbb{N}$ and $\alpha\in(0,1)$.\\

\underline{$1^{st}$ step, proof of ``$\supseteq$"}: similarly to Proposition \ref{prop:Holder_suff}, we define constants $C_a$ and $C_b$ as in 
\eqref{eq:coeff_bound} and make use of the estimates in \eqref{es:f0} and \eqref{es:g} to conclude that $f\in L^\infty(\mathbb{R})$. The next step is to show the existence of the $n$-th derivative  $g^{(n)}$ of $g$ in \eqref{def:f0_g}. This follows by uniform convergence since, for every $x\in\mathbb{R}$ and $0\leq \ell \leq n<r$, by (V), we have
\[
\left|\;\sum\limits_{j\in\mathbb{N}}\;\sum\limits_{k\in\mathbb{Z}}\;b_{j,k}\;\psi^{(\ell)}_{j,k}(x)\right| \;\leq\; \sum\limits_{j\in\mathbb{N}}\;\sum\limits_{k\in\mathbb{Z}}\;|b_{j,k}|\;|\psi^{(\ell)}_{j,k}(x)|\;\leq\; C_b\;C_\Gamma\;C_{sm,\ell}\;\sum\limits_{j\in\mathbb{N}}\; 2^{-j(r-\ell)}\;<\;\infty,
\]
The same argument as in  Proposition \ref{prop:Holder_suff} leads to $g^{(n)}\in\mathcal{C}^\alpha(\mathbb{R})$ and, thus, $f\in\mathcal{C}^r(\mathbb{R})$.

\underline{$2^{nd}$ step, proof of ``$\subseteq$" resembles \cite{MR1083586}}: similarly to Proposition \ref{prop:Holder_nec}, we  consider $f\in B^r_{\infty,\infty}(\mathbb{R})\cap L^2(\mathbb{R})$ and the uniform bound for $|\langle f, \phi_k\rangle|$ is obtained as in \eqref{eq:coarse_bound}. Exploiting the first $n$ vanishing moments of the tight frame we have
\[\begin{array}{rcl}
\left|\langle f,\psi_{j,k}\rangle\right| &=& \left|\;\int_{\mathbb{R}}\;f(x)\;\psi_{j,k}(x)\;dx\;\right| \\ \\
&=& \left|\;\int_{\mathbb{R}}\;\left(\;f(x)-f(x_{j,k})-\sum\limits_{\ell=1}^{n-1}\frac{f^{(\ell)}(x_{j,k})}{\ell!}(x-x_{j,k})^\ell\;\right)\;\psi_{j,k}(x)\;dx\;\right|\;,
\end{array}\]
where the $x_{j,k}$ are as in (IV). Using the property of the Taylor expansion of $f$ centered in $x_{j,k}$ with the Lagrange remainder term, we have that, for every $x\in\mathbb{R}$, there exists a measurable $\xi(x)\in\mathbb{R}$, with $|\xi(x)-x_{j,k}|\leq|x-x_{j,k}|$, such that
\[\left|\langle f,\psi_{j,k}\rangle\right|\;\leq\;\left|\;\int_{\mathbb{R}}\;\frac{f^{(n)}(\xi(x))}{n!}(x-x_{j,k})^n\;\psi_{j,k}(x)\;dx\;\right|\;.\]
Now we can exploit $n+1$ vanishing moments, the H\"older regularity $\alpha$ of $f^{(n)}$ and (IV) to get
\begin{equation}\label{(18a)}
\begin{array}{rcl}
\left|\langle f,\psi_{j,k}\rangle\right| &\leq& \frac{1}{n!}\left|\;\int_{\mathbb{R}}\;\left(\;f^{(n)}(\xi(x))-f^{(n)}(x_{j,k})\;\right)(x-x_{j,k})^n\;\psi_{j,k}(x)\;dx\;\right| \\ \\
&\leq&  \frac{\|f^{(n)}\|_{B^\alpha_{\infty,\infty}}}{n!}\;\int_{\mathbb{R}}\;\left|\xi(x)-x_{j,k}\right|^\alpha\;|x-x_{j,k}|^n\;|\psi_{j,k}(x)|\;dx \\ \\
&\leq&  \frac{\|f^{(n)}\|_{B^\alpha_{\infty,\infty}}}{n!}\;\int_{\mathbb{R}}\;|x-x_{j,k}|^{r}\;|\psi_{j,k}(x)|\;dx \\ \\
&\leq&  \frac{\|f^{(n)}\|_{B^\alpha_{\infty,\infty}}}{n!}\;\int_{\mathbb{R}}\;|y|^{r}\;|\psi_{j,k}(y+x_{j,k})|\;dy\;\leq\; 2^{-j\left(r+\frac{1}{2}\right)}\frac{C_{vm,r}\;\|f^{(n)}\|_{B^\alpha_{\infty,\infty}}}{n!}.
\end{array}
\end{equation}
Thus, the claim follows.
\end{proof}

\begin{rem}
If $s\in\mathbb{N}$ and $v\geq s$ in Theorem \ref{thm:Holder_gen}, then the wavelet tight frame $\mathcal{F}$ does not need to belong to $\mathcal{C}^s(\mathbb{R})$. It suffices to have
$\mathcal{F} \subset \mathcal{C}^{s-1}(\mathbb{R})$ with the $(s-1)$-st derivatives of its elements being Lipschitz-continuous.
\end{rem}

\section{Characterization of H\"older-Zygmund spaces $B^r_{\infty,\infty}(\mathbb{R})$, $r\in\mathbb{N}$} \label{sec:Zyg}

It is well known \cite{MR1228209} that the H\"older spaces with integer H\"older exponents cannot
be characterized via either a wavelet or a wavelet tight frame system $\mathcal{F}$ .
Indeed, if $r=1$, the estimate \eqref{eq:A_bound} does not follow from \eqref{eq:coeff_bound}.
Thus, similarly to Theorem \ref{thm:wave_char}, the natural spaces in this context are the
H\"older-Zygmund spaces $B^{r}_{\infty,\infty}(\mathbb{R})$ for $r \in \mathbb{N}$.
This section is devoted to the proof of the wavelet tight frame characterization of such spaces, 
see Theorem \ref{thm:Zygmund}. The results of  Theorems \ref{thm:Holder_gen} and \ref{thm:Zygmund} yield Theorem \ref{thm:main}.

\begin{thm} \label{thm:Zygmund} Let $s>0$ and $v\in\mathbb{N}$.
Assume $\mathcal{F} \subset \mathcal{C}^s(\mathbb{R})$ satisfies (I)-(V) with $v$ vanishing moments.
Then, for $0<r<\min(s,v)$, $r\in\mathbb{N}$,
	\[
	B^{r}_{\infty,\infty}(\mathbb{R})\;=\; \Big\{\; \sum_{k\in\mathbb{Z}}a_k\phi_k+\sum_{j\in\mathbb{N}}\sum_{k\in\mathbb{Z}}b_{j,k}\psi_{j,k}\;: \; (a,b) \in \ell_{\infty,\infty}^r \ \hbox{with} \ \ a=\{a_k\}_{k \in \mathbb{Z}}, \ b=\{b_{j,k}\}_{j\in\mathbb{N}, k \in \mathbb{Z}} \Big\}.
\]
\end{thm}

The significant case is when $r=1$, since for all other integers one usually argues similarly to the proof of Theorem \ref{thm:Holder_gen}. 
In the case $r=1$, for the inclusion $\supseteq$ we
use the argument similar to the one in Proposition \ref{prop:Holder_suff}. On the other hand, for the inclusion $\subseteq$, we cannot exploit the vanishing moments as done in \eqref{eq:vm_trick}. To circumvent this problem, inspired by \cite{MR1289147}, we consider an auxiliary orthogonal wavelet system which satisfies the assumptions of Theorem \ref{thm:wave_char}. This way we get a convenient expansion for $f\in B^r_{\infty,\infty}(\mathbb{R})$ and make use of the wavelet characterization of $B^{\alpha}_{\infty,\infty}(\mathbb{R})$ for $\alpha\in(r,s)\setminus\mathbb{N}$ in Theorem \ref{thm:Holder_gen}.

\begin{proof}
We only prove the claim for $r=1<\min(s,v)$. In this case $B^r_{\infty,\infty}(\mathbb{R})=\Lambda(\mathbb{R})\cap L^\infty(\mathbb{R})$. The general case follows using an argument similar to the one in the proof of Theorem \ref{thm:Holder_gen}. \\

\underline{$1^{st}$ step, proof of ``$\supseteq$"}: similarly to Proposition \ref{prop:Holder_suff}, we define constants $C_a$ and $C_b$ as in 
\eqref{eq:coeff_bound} and make use of the estimates in \eqref{es:f0} and \eqref{es:g} to conclude that $f\in L^\infty(\mathbb{R})$. 
Let $x\in\mathbb{R}$. It suffice to consider $h>0$. Then, for $r=1$, we obtain
\[\begin{array}{l}
\left|\;g(x+h)\;-\;2\;g(x)\;+\;g(x-h)\;\right| \;\leq\; \\ \\
\qquad \leq\; \sum\limits_{j\in\mathbb{N}}\;\sum\limits_{k\in\mathbb{Z}}\;\left|\;b_{j,k}\;\right|\;\left|\;\psi_{j,k}(x+h)\;-\;2\;\psi_{j,k}(x)\;+\;\psi_{j,k}(x-h)\;\right| \\ \\
\qquad \le\; C_b \, h\;\sum\limits_{j\in\mathbb{N}}\;\frac{2^{-j\frac{3}{2}}}{h}\;\sum\limits_{k\in\mathbb{Z}}\;\left|\;\psi_{j,k}(x+h)\;-\;2\;\psi_{j,k}(x)\;+\;\psi_{j,k}(x-h)\;\right|\;.
\end{array}\]
Since there exists $J\in\mathbb{Z}$ such that
\[2^{-J}\;<\;h\;\leq\;2^{-J+1}\;,\]
we have
\[\begin{array}{l}
\left|\;g(x+h)\;-\;2\;g(x)\;+\;g(x-h)\;\right| \;\leq\; \\ \\
\qquad \leq\; C_bh\;\sum\limits_{j\in\mathbb{N}}\;2^{J-j\frac{3}{2}}\;\sum\limits_{k\in\mathbb{Z}}\;\left|\;\psi_{j,k}(x+h)\;-\;2\;\psi_{j,k}(x)\;+\;\psi_{j,k}(x-h)\;\right|.\\ \\
\end{array}\]
To estimate $\left|\;g(x+h)\;-\;2\;g(x)\;+\;g(x-h)\;\right|$, we consider
\[\begin{array}{c}
A\;=\;\sum\limits_{j=1}^{J-1}\;2^{J-j\frac{3}{2}}\;\sum\limits_{k\in\mathbb{Z}}\;\left|\;\psi_{j,k}(x+h)\;-\;2\;\psi_{j,k}(x)\;+\;\psi_{j,k}(x-h)\;\right|\quad\textrm{and} \\ \\
B\;=\;\sum\limits_{j=J}^{\infty}\;2^{J-j\frac{3}{2}}\;\sum\limits_{k\in\mathbb{Z}}\;\left|\;\psi_{j,k}(x+h)\;-\;2\;\psi_{j,k}(x)\;+\;\psi_{j,k}(x-h)\;\right|.
\end{array}\]
If $J\leq 1$, $A=0$. Otherwise, since the tight frame $\cal{F}$ belongs to $\mathcal{C}^s(\mathbb{R})$, $s>1$, we use the mean value theorem twice for every framelet and find {
$\xi_{j,k}(x)\in[x,x+h]$ and $\eta_{j,k}(x)\in[x-h,x]$ such that
\[\begin{array}{rcl}
\left|\;\psi_{j,k}(x+h)\;-\;2\;\psi_{j,k}(x)\;+\;\psi_{j,k}(x-h)\;\right|&=& \left|\;\psi_{j,k}(x+h)\;-\;\psi_{j,k}(x)\;-\;\left(\;\psi_{j,k}(x)\;-\;\psi_{j,k}(x-h)\;\right)\;\right| \\ \\
&=& h\; \left|\;\psi'_{j,k}(\xi_{j,k}(x))\;-\;\psi'_{j,k}(\eta_{j,k}(x))\;\right|.
\end{array}\]
Now, for $\epsilon>0$ with $r=1<1+\epsilon<s$, using (V) we get
\[\begin{array}{rcl}
\left|\;\psi_{j,k}(x+h)\;-\;2\;\psi_{j,k}(x)\;+\;\psi_{j,k}(x-h)\;\right|&=& C_{sm,1+\epsilon}\;2^{j\left(\epsilon+\frac{3}{2}\right)}\;h\; |\;\xi_{j,k}(x)\;-\;\eta_{j,k}(x)\; |^{\epsilon} \\ \\
&\leq&C_{sm,1+\epsilon}\;2^{j\left(\epsilon+\frac{3}{2}\right)+\epsilon}\;h^{1+\epsilon} \;\leq\; C_{sm,1+\epsilon}\;2^{(j-J)(1+\epsilon)+\frac{j}{2}+1+2\epsilon}.
\end{array}\]
Moreover, the sum in $A$ over $k$ has at most
\[|\Gamma_j(x+h)|\;+\;|\Gamma_j(x-h)|\;+\;|\Gamma_j(x)|\;\leq\;3\;C_\Gamma\]
non-zero summands and, thus, we get
\[A\;\leq\;C_\Gamma\;C_{sm,1+\epsilon}\;2^{1+2\epsilon}\;3\;\sum\limits_{j=1}^{J-1}\;(2^{-\epsilon})^{J-j}.\]
Since $\epsilon>0$, $A$ is bounded. To conclude the proof, we observe that
\[B\;\leq\;C_\Gamma\;C_{sm,0}\;3\;\sum\limits_{j=J}^{\infty}\;2^{J-j}\;=\;6\;C_\Gamma\;C_{sm,0}.\]}
Thus,
$g\in B^1_{\infty,\infty}(\mathbb{R})$ and, therefore, $f\in B^1_{\infty,\infty}(\mathbb{R})$ with $\|f\|_{B_{\infty,\infty}^1} \le C \|(a,b)\|_{\ell_{\infty, \infty}^1}$ for some constant $C>0$.

\underline{$2^{nd}$ step, proof of ``$\subseteq$"}: similarly to Proposition \ref{prop:Holder_nec} we only consider 
$f\in \Lambda(\mathbb{R})\cap L^2(\mathbb{R})$. The uniform bound for $|\langle f, \phi_k\rangle|$ is obtained similarly to \eqref{eq:coarse_bound}. 
To obtain the bound for $|\langle f,\psi_{j,k}\rangle|$, we 
let $\widetilde{\Phi}\cup\{\widetilde{\Psi}_\ell\}_{\ell\in\mathbb{N}}\subseteq \mathcal{C}^s(\mathbb{R})$ be an auxiliary compactly supported orthogonal wavelet system with $v$ vanishing moments (e.g. Daubechies $2n$-tap wavelets \cite{MR1162107} with large enough $n\in\mathbb{N}$). $\widetilde{\Phi}\cup\{\widetilde{\Psi}_\ell\}_{\ell\in\mathbb{N}}$ satisfies the assumptions of Theorem \ref{thm:wave_char} and fulfills (I)-(V), with appropriate
$\widetilde{\Gamma}_j$ and $\widetilde{C}_{supp}>0$, $\widetilde{C}_\Gamma>0$, $\widetilde{C}_{vm,r}>0$ and $\widetilde{C}_{sm,r}>0$.
Then, from \eqref{eq:TF}, we have $f(x)\;=\;\widetilde{f}_0(x)\;+\;\widetilde{g}(x)$, $x\in\mathbb{R}$, where
\[
\widetilde{f}_0(x)\;=\;\sum\limits_{m\in\mathbb{Z}}\;\langle f,\widetilde{\phi}_m\rangle\;\widetilde{\phi}_m(x)\quad\textrm{and}\quad\widetilde{g}(x)\;=\;\sum\limits_{\ell\in\mathbb{N}}\;\sum\limits_{m\in\mathbb{Z}}\;\langle f,\widetilde{\psi}_{\ell,m}\rangle\;\widetilde{\psi}_{\ell,m}(x).
\]
Thus, for every $j\in\mathbb{N}$ and $k\in\mathbb{Z}$, we get
\[
|\langle f,\psi_{j,k}\rangle|\;\leq\;\left|\langle\;\widetilde{f}_0,\;\psi_{j,k}\;\rangle\right|\;+\;\sum\limits_{\ell\in\mathbb{N}}\;\sum\limits_{m\in\mathbb{Z}}\;|\langle f,\widetilde{\psi}_{\ell,m}\rangle|\;|\langle \widetilde{\psi}_{\ell,m}, \psi_{j,k}\rangle|\;.
\]
Let $\alpha\in(1,s)\setminus\mathbb{N}$. Since both tight frames belong to $\mathcal{C}^s(\mathbb{R})\subseteq\mathcal{C}^\alpha(\mathbb{R})$, the function $\widetilde{f}_0$, which is locally the finite sum of $\mathcal{C}^\alpha$-functions,
belongs to $\mathcal{C}^\alpha(\mathbb{R})$, and, by Theorem \ref{thm:Holder_gen}, there exists $C_1>0$, such that
\[
\sup_{k\in\mathbb{Z}}|\langle \widetilde{f}_0,\psi_{j,k}\rangle|\;\leq\;C_1\;2^{-j\left(\alpha+\frac{1}{2}\right)}\;\leq\;C_1\;2^{-j\frac{3}{2}},\quad j\in\mathbb{N}.
\]
Moreover, by Theorem \ref{thm:wave_char} and due to $f \in \Lambda(\mathbb{R})$, there exists $C_2>0$ such that
\[\sup_{m\in\mathbb{Z}}|\langle f,\widetilde{\psi}_{\ell,m}\rangle|\;\leq\;C_2\;2^{-\ell\frac{3}{2}}\;,\quad \ell\in\mathbb{N}.\]
Thus,{
\begin{equation}\label{eq:Zyg1}\begin{array}{rcl}
|\langle f,\psi_{j,k}\rangle|&\leq&C_1\;2^{-j\frac{3}{2}}\;+\;C_2\;\sum\limits_{\ell\in\mathbb{N}}\;2^{-\ell\frac{3}{2}}\;\sum\limits_{m\in\mathbb{Z}}\;|\langle \widetilde{\psi}_{\ell,m}, \psi_{j,k}\rangle|\\ \\
&=& C_1\;2^{-j\frac{3}{2}}\;+\;C_2\left(\sum\limits_{\ell=1}^{j-1}\;2^{-\ell\frac{3}{2}}\;\sum\limits_{m\in\mathbb{Z}}\;|\langle \widetilde{\psi}_{\ell,m}, \psi_{j,k}\rangle|\;+\;\sum\limits_{\ell=j}^{\infty}\;2^{-\ell\frac{3}{2}}\;\sum\limits_{m\in\mathbb{Z}}\;|\langle \widetilde{\psi}_{\ell,m}, \psi_{j,k}\rangle|\right) \\ \\
&=& C_1\;2^{-j\frac{3}{2}}\;+\;C_2\;\left(\;A\;+\;B\;\right).
\end{array}\end{equation}
The sums in \eqref{eq:Zyg1} over $m$ have at most
\[|\widetilde{\Gamma}_\ell(\hbox{supp}(\psi_{j,k}))|\;\leq\;\widetilde{C}_{\Gamma}\;(\;2^\ell\;|\hbox{supp}(\psi_{j,k})|\;+\;1\;)\;\leq\;\widetilde{C}_{\Gamma}\;(\;C_{supp}\;2^{\ell-j}\;+\;1\;)\]
non-zero summands. When $\ell < j$,  by assumption (V) for $\widetilde{\psi}_{\ell,m}$ and \eqref{(18a)} with 
$f=\widetilde{\psi}_{\ell,m}$, due to Theorem \ref{thm:Holder_gen}, we have
\begin{eqnarray*}
|\langle \widetilde{\psi}_{\ell,m},\psi_{j,k}\rangle| \leq  C_{vm, \alpha} \, 2^{-j(\alpha+\frac{1}{2})} \|\widetilde{\psi}_{\ell,m}\|_{\cal{C}^\alpha}
\le \,C_{vm,\alpha} \, \widetilde{C}_{sm,\alpha} \;2^{(\ell-j)\left(\alpha+\frac{1}{2}\right)},
\end{eqnarray*}
uniformly in $m$ and $k$. Thus, substituting  $\ell'=j-\ell$, we obtain
\begin{equation}\label{eq:Zyg3}
\begin{array}{rcl}
A &\leq& C_{vm,\alpha}\;\widetilde{C}_\Gamma\;\widetilde{C}_{sm,\alpha}\;\sum_{\ell=1}^{j-1}\;2^{-\ell\frac{3}{2}}\;2^{(\ell-j)\left(\alpha+\frac{1}{2}\right)}\;(\;C_{supp}\;2^{\ell-j}\;+\;1\;)\\ \\
&=& C_{vm,\alpha}\;\widetilde{C}_\Gamma\;\widetilde{C}_{sm,\alpha}\;2^{-j\frac{3}{2}}\; \sum_{\ell'=1}^{j-1}\;2^{-\ell'\left(\alpha-1\right)}\;(\;C_{supp}\;2^{-\ell'}\;+\;1\;)\\ \\
&\leq&  C_3\;2^{-j\frac{3}{2}},
\end{array}
\end{equation}
for some $C_3>0$, due to the fact that $\alpha>1$. \\
On the other hand, when $\ell\geq j$, using (II) and (V), we get
\[
 |\langle \widetilde{\psi}_{\ell,m},\psi_{j,k}\rangle|\;\leq\;C_{sm,0}\;\widetilde{C}_{sm,0}\;2^{(j+\ell)/2}\;\min\left(C_{supp}2^{-j},\widetilde{C}_{supp}2^{-\ell}\right)\;=\;C_4\;2^{(j-\ell)/2},
\]
uniformly in $m$ and $k$. Thus, after the substitution $\ell'=\ell-j$, we obtain
\begin{equation}\label{eq:Zyg4}
\begin{array}{rcl}
B &\leq& C_4\;\widetilde{C}_\Gamma\;\sum_{\ell=j}^{\infty}\;2^{-\ell\frac{3}{2}}\;(\;C_{supp}\;2^{\ell-j}\;+\;1\;)\;2^{(j-\ell)/2}\\ \\
&=& C_4\;\widetilde{C}_\Gamma\;2^{-j\frac{3}{2}}\; \sum_{m=0}^{\infty}\;2^{-2m}\;(\;C_{supp}\;2^{m}\;+\;1\;) \le  C_5\;2^{-j\frac{3}{2}}
\end{array}
\end{equation}} 
for some constant $C_5>0$. Combining \eqref{eq:Zyg1}, \eqref{eq:Zyg3} and \eqref{eq:Zyg4} we finally get
\[
\sup_{k\in\mathbb{Z}}|\langle f,\psi_{j,k}\rangle|\;\leq\;\left( C_1+C_2 C_3+C_2 C_5\right)\;2^{-j\frac{3}{2}}\;,\quad  j\in\mathbb{N}.
\]
Thus, the claim follows, i.e., there exists a constant $C>0$ such that $\|f\|_{B_{\infty,\infty}^1} \ge C \|(a,b)\|_{\ell_{\infty, \infty}^1}$.
\end{proof}

\begin{rem} \label{rem:norm_eq} The norm equivalence between the Besov norm $\|\cdot\|_{B_{\infty, \infty}^r}$ and $\|\cdot\|_{\ell_{\infty,\infty}^r}$, $r \in (0,\infty)$,
is a consequence of Theorem \ref{thm:main} and the Open Mapping Theorem.
\end{rem}

\section{H\"older-Zygmund regularity of semi-regular subdivision} \label{sec:num}

In this section, we show how to apply Theorem \ref{thm:main} for estimating  the H\"older-Zygmund regularity of a semi-regular subdivision limit  
from the decay of its inner products (frame coefficients) with respect to a given tight frame $\mathcal{F}$ satisfying (I)-(V) for some $s>0$ and $v\in\mathbb{N}$.
In Subsection \ref{subsec:reg_estimates},  in a 
general irregular setting, we discuss how to obtain such regularity estimates using the result of Theorem \ref{thm:main}. In
Subsection \ref{subsec:coefficients}, we introduce a method for computing the frame coefficients in the semi-regular case. In Subsection \ref{subsec:numerics}, we illustrate our
results with examples of semi-regular B-spline, Dubuc-Deslauriers subdivision and semi-regular interpolatory schemes based on radial basis functions.
The latter example in the regular setting reduces to the construction in  \cite{MR2231695}.

\subsection{Optimal H\"older-Zygmund exponent: two methods for its estimation} \label{subsec:reg_estimates}
\begin{defn} \label{defn:opt_exp}
Let $f\in L^\infty(\mathbb{R})$. We call \emph{optimal H\"older-Zygmund (regularity) exponent of $f$} the real number
\[r(f)\;=\;\sup\{\;r>0\;:\;f\in B^r_{\infty,\infty}(\mathbb{R})\;\}.\]
\end{defn}

Assume that $r(f) \in (0,\min(s,v))$ and that we are given 
\[\gamma_j\;=\;\sup_{k\in\mathbb{Z}}\;|\langle f,\psi_{j,k}\rangle|,\quad j\in\mathbb{N}.\]
By Theorem \ref{thm:main}, for every $\epsilon>0$, there exists a constant $C_\epsilon>0$ such that, for every $j\in\mathbb{N}$,
\begin{equation}\label{eq:lin_reg}
\gamma_j\;\leq\;C_\epsilon\;2^{-j\left(r(f)-\epsilon+\frac{1}{2}\right)},\quad\textrm{i.e.}\quad j\left(r(f)-\epsilon+\frac{1}{2}\right)\;-\;\log_2(C_\epsilon)\;\leq\;-\log_2(\gamma_j).
\end{equation}
From \eqref{eq:lin_reg} we infer that searching for $r(f)$ is equivalent to searching for the largest slope of a line lying under the set of points $\{(j,-\log_2(\gamma_j))\}_{j\in\mathbb{N}}$. With this interpretation in mind, the natural approach (see e.g. \cite{MR1162107}) to approximate $r(f)$ is to compute the real-valued sequence $\{r_n(f)\}_{n\in\mathbb{N}}$, where $r_n(f)-\frac{1}{2}$ is the slope of the regression line for the points $\{(j,-\log_2(\gamma_j))\}_{j=1}^{n+1}$. This method is robust, i.e. for larger $n$  the contributions of the levels $j \ge n$ become less significant, thus, the difference between $r_n(f)$ and $r_{n+1}(f)$ is small and we are able to estimate the overall distribution of $\{(j,-\log_2(\gamma_j))\}_{j\in\mathbb{N}}$. However, 
examples in Subsection \ref{subsec:numerics} illustrate that the convergence of $\{r_n(f)\}_{n\in\mathbb{N}}$ 
towards $r(f)$ is very slow. One of the main reasons for such a behavior is the value of the unknown $C_\epsilon$, which can be significant, e.g when $f\not\in B^{r(f)}_{\infty,\infty}(\mathbb{R})$.

An alternative approach for estimating the H\"older-Zygmund exponent is given by the following Proposition.

\begin{prop} \label{prop:my_test} Let $r(f)$ be the optimal H\"older-Zygmund exponent of $f\in {\cal C}^0(\mathbb{R})$.
If  
$$\; \; 0< r^*(f)\;=\;\lim_{n\rightarrow\infty}\;\log_2\left(\frac{\gamma_n}{\gamma_{n+1}}\right)\;-\;\frac{1}{2}\; < \; \min(s,v)\;,
$$
then $r^*(f)=r(f)$.
\end{prop}

\begin{proof}
We first prove that $r^*(f)\leq r(f)$ and then, by contradiction, that $r^*(f)=r(f)$. 

\noindent Let $\epsilon>0$. We consider the series
\begin{equation} \label{eq:sum_test}
S(r^*(f)-\epsilon)\;=\;\sum_{j=1}^\infty\;2^{j\left(r^*(f)-\epsilon+\frac{1}{2}\right)}\;\gamma_j.
\end{equation}
By the assumption, we obtain
\begin{equation} \label{eq:rat_test}
\lim_{n\rightarrow \infty}\;\frac{2^{(n+1)\left(r^*(f)-\epsilon+\frac{1}{2}\right)}\;\gamma_{n+1}}{2^{n\left(r^*(f)-\epsilon+\frac{1}{2}\right)}\;\gamma_n}\;
=\; 2^{r^*(f)-\epsilon+\frac{1}{2}}\, \lim_{n\rightarrow \infty}\;\frac{\gamma_{n+1}}{\gamma_n}=2^{-\epsilon}\;.
\end{equation}
Thus, by the ratio test, the series $S(r^*(f)-\epsilon)$ in \eqref{eq:sum_test} converges for every $\epsilon>0$. Consequently,  the non-negative
summands of $S(r^*(f)-\epsilon)$ are uniformly bounded, i.e. there exists $C_\epsilon>0$ such that
\[
\gamma_j\;\leq\;C_\epsilon2^{-j\left(r^*(f)-\epsilon+\frac{1}{2}\right)}, \quad j \in \mathbb{N}.
\]
Therefore, by Definition \ref{defn:opt_exp} and by \eqref{eq:lin_reg}, $r^*(f)\leq r(f)$. 

\noindent On the other hand, by \eqref{eq:lin_reg}, if $r^*(f)<r(f)$, then there exists $\delta>0$ and a cosnstant $C_\delta>0$ such that
\begin{equation*}
  \gamma_j\;\leq\;C_\delta\, 2^{-j\left(r^*(f)+\delta+\frac{1}{2}\right)}, \quad j \in \mathbb{N}.
\end{equation*}
Therefore, similarly to \eqref{eq:rat_test}, we obtain that the series $S(r^*(f)+\delta/2)$ diverges and at the same time is majorized by
\[
  S(r^*(f)+\delta/2)\;=\;\sum_{j=1}^\infty\;2^{j\left(r^*(f)+\frac{\delta}{2}+\frac{1}{2}\right)}\;\gamma_j\;\leq\;C_\delta
	\;\sum_{j=1}^\infty\; 2^{-j\delta/2}\;.
\]
Thus, due to this contradiction, $r^*(f)=r(f)$.
\end{proof}

The advantage of the approach in Proposition \ref{prop:my_test} is that it eliminates the effect of the constant $C_\epsilon$ in
\eqref{eq:lin_reg}. Even though the existence of $r^*(f)$ is not guaranteed and the elements of the sequence
\[r^*_n(f)\;=\;\log_2\left(\frac{\gamma_n}{\gamma_{n+1}}\right)\;-\;\frac{1}{2},\quad n\in\mathbb{N},\]
can oscillate wildly, our numerical experiments in Subsection \ref{subsec:numerics} provide examples 
which illustrate the cases when $\{r^*_n(f)\}_{n \in \mathbb{N}}$ converges to $r(f)$ rapidly.
The convergence in these examples is much faster than that of the linear regression method.

\begin{rem}
The series in \eqref{eq:sum_test} with $\epsilon=0$ becomes
\[S(r)\;=\;\left\|\;\left\{\;2^{j\left(r+\frac{1}{2}\right)}\;\left\|\;\left\{\;\langle\; f,\;\psi_{j,k}\;\rangle\;\right\}_{k\in\mathbb{Z}}\;\right\|_{\ell^\infty}\;\right\}_{j\in\mathbb{N}}\;\right\|_{\ell^1}.\]
This norm appears in the characterization of $B^r_{\infty,1}(\mathbb{R})$ in Theorem \ref{thm:wave_char} and correspond to
$(a,b) \in \ell_{\infty,1}^r$, $r \in (0,\infty)$. Even if the case
 $p=\infty$ and $q=1$ is not covered by Theorem \ref{thm:main}, this observation is consistent with  $B^r_{p,q_1}(\mathbb{R})\subseteq B^r_{p,q_2}(\mathbb{R})$ for $q_1\leq q_2$.
\end{rem}

\subsection{Computation of frame coefficients: semi-regular case} \label{subsec:coefficients}

Assumptions (I)-(V) do not require the semi-regularity of the mesh and all the above results hold even in the irregular case. To use the presented results in practice, however, we need  an efficient method for computing the frame coefficients 
$$
 \{a_k=\langle f,\phi_k \rangle \}_{k\in\mathbb{N}} \quad \hbox{and} \quad \{b_{j,k}=\langle f,\psi_{j,k} \rangle\}_{j \in \mathbb{N}, k\in\mathbb{Z}}. 
$$ 
If the functions $\phi_k$ or $\psi_{j,k}$ are defined implicitly as limits of irregular subdivision, suitable quadrature rules are not 
available in the literature. In what follows, we focus on semi-regular refinable functions, which are e.g. \cite{MR1687779,MR1356994} 
generated by subdivision  schemes on meshes of the type \eqref{eq:mesh}. Such
convergent iterative processes define compactly supported, uniformly continuous basic limit functions $\{\zeta_k(x)\; : \; k\in\mathbb{Z}\}$
with the properties
\begin{itemize}
\item[(i)] there exists a bi-infinite matrix $\mathbf{Z}$ such that the bi-infinite column vector $[\zeta_k(x)\; :\; k\in\mathbb{Z}]$
satisfies
\begin{equation} \label{eq:ref_eq}
[\zeta_k(x)\; :\; k\in\mathbb{Z}]\;=\; (\mathbf{Z}^j)^T \;[\zeta_k(2^jx)\; :\; k\in\mathbb{Z}],\quad x\in\mathbb{R},\quad j\in\mathbb{N};
\end{equation}
\item[(ii)] $1$ is the unique largest eigenvalue of $\mathbf{Z}$ in absolute value with the corresponding eigenvector  $\mathbf{1}$
all of whose entries are equal to $1$ 
\begin{equation} \label{eq:POU}
\mathbf{1}\;=\;\mathbf{Z}\;\mathbf{1}\qquad\textrm{and}\qquad[\zeta_k(2^jx)\; :\; k\in\mathbb{Z}]^T\;\mathbf{1}\;\equiv\;1;
\end{equation}
\item[(iii)] there exist $k_\ell(\mathbf{Z}),\;k_r(\mathbf{Z})\in\mathbb{Z}$, $k_\ell(\mathbf{Z})\leq k_r(\mathbf{Z})$ and bi-infinite vectors $\mathbf{z}_\ell,\;\mathbf{z}_r$ such that
\[\zeta_k(x)\;=\;\left\{\begin{array}{cl}
\zeta_{k_\ell(\mathbf{Z})}(\;x\;+\;h_\ell\;(k_\ell(\mathbf{Z})-k)\;) & \textrm{ for } k \leq k_\ell(\mathbf{Z}) \\ \\
\zeta_{k_r(\mathbf{Z})}(\;x\;+\;h_r\;(k_r(\mathbf{Z})-k)\;) & \textrm{ for } k \geq k_r(\mathbf{Z})
\end{array}\right.,\quad x\in\mathbb{R},\]
and
\[\mathbf{Z}(i,k)\;=\;\left\{\begin{array}{cl}
\mathbf{z}_\ell(\;i\;+\;2\;(k_\ell(\mathbf{Z})-k)\;) & \textrm{ for } k \leq k_\ell(\mathbf{Z}) \\ \\
\mathbf{z}_r(\;i\;+\;2\;(k_r(\mathbf{Z})-k)\;) & \textrm{ for } k \geq k_r(\mathbf{Z})
\end{array}\right.,\quad i\in\mathbb{Z}.\]
\end{itemize}
A wavelet tight frame $\mathcal{F}$ is based on another semi-regular subdivision scheme with limit functions  
$\left\{\varphi_k\; : \; k\in\mathbb{Z}\right\}$ and bi-infinite subdivision matrix $\mathbf{W}$ such that 
the assumptions (i)-(iii) are satisfied.  The framelets $\Psi_j$ that belong to $\mathcal{F}$ are constructed from the re-normalized
bi-infinite column vector 
\begin{equation} \label{eq:renorm}
\Phi\;=\;\mathbf{D}^{-1/2}\;[\varphi_k\; : \; k\in\mathbb{Z}]\quad\textrm{with} \quad \mathbf{D}\;=\;\diag\left(\mathbf{d}\right),\quad\mathbf{d}(k)\;=\;\int_\mathbb{R}\;\varphi_k(x)\;dx,\;k\in\mathbb{Z}.
\end{equation} 
via the the bi-infinite matrix $\mathbf{Q}$ through the relation
\begin{equation} \label{eq:WTF}
\Psi_j\;=\;2^{j/2}\;\mathbf{Q}^T\;\Phi(2^j\cdot)\quad\hbox{i.e.}\quad\Psi_j\;=\;2^{\frac{j-1}{2}}\;\Psi_1(2^{\frac{j-1}{2}}\cdot),\quad j\in\mathbb{N},
\end{equation}
see \cite{MR2110512} for more details. Now we are ready to present an algorithm for determining the frame 
coefficients of the basic limit functions in $[\zeta_k\; : \;k\in\mathbb{Z}]$.

\begin{prop} \label{prop:coeff_comp}
For every $j\in\mathbb{N}$, $i,\;k\in\mathbb{Z}$, we have $<\zeta_i,\psi_{j,k}>\;=\;\mathbf{C}_j(i,k)$, where
\[\mathbf{C}_j\;=\;2^{-j/2}\;(\mathbf{Z}^j)^T\;\mathbf{G}\;\mathbf{D}^{-1/2}\;\mathbf{Q},\]
 with the cross-Gramian $\mathbf{G} = \left(\mathbf{G}(i,k) \right)\;=\;\left(
 \langle\zeta_i,\varphi_k \rangle\right)$.
\end{prop}

\begin{proof}
Applying \eqref{eq:WTF}, the substitution $y=2^j x$, \eqref{eq:ref_eq} and \eqref{eq:renorm}, we get
\[\begin{array}{rcl}
\int_{\mathbb{R}}\;[\zeta_i(x)\; :\; i\in\mathbb{Z}]\;\Psi_j(x)^T\;dx&=&2^{j/2}\;\int_{\mathbb{R}}\;[\zeta_i(x)\; : \; k\in\mathbb{Z}]\;\Phi(2^j x)^T\;dx\;\mathbf{Q} \\ \\
&=& 2^{-j/2}\;\int_{\mathbb{R}}\;[\zeta_i(2^{-j}y)\; : \; i\in\mathbb{Z}]\;\Phi(y)^T\;dy\;\mathbf{Q} \\ \\
&=& 2^{-j/2}\;(\mathbf{Z}^{j})^T\;\int_{\mathbb{R}}\;[\zeta_i(y) \; : \; i\in\mathbb{Z}]\;\Phi(y)^T\;dy\;\mathbf{Q} \\ \\
&=& 2^{-j/2}\;(\mathbf{Z}^{j})^T\;\int_{\mathbb{R}}\;[\zeta_i(y)\; : \; i\in\mathbb{Z}]\;[\varphi_k(y)\; :\; k\in\mathbb{Z}]^T\;dy\;\mathbf{D}^{-1/2}\;\mathbf{Q} \\ \\
&=& 2^{-j/2}\;(\mathbf{Z}^{j})^T\;\mathbf{G}\;\mathbf{D}^{-1/2}\;\mathbf{Q}.
\end{array}\]
\end{proof}
To determine the cross-Gramian $\mathbf{G}$ in Proposition \ref{prop:coeff_comp} we use a strategy similar to the one in \cite{first_paper} which
yields the linear system 
\[
\mathbf{G}\;=\;\frac{1}{2}\;\mathbf{Z}^T\;\mathbf{G}\;\mathbf{P}.
\]
See \cite{first_paper} for more details.

\subsection{Numerical estimates } \label{subsec:numerics}

For simplicity of presentation, in this subsection we choose $h_\ell=1$ and $h_r=2$ in \eqref{eq:mesh}. The tight wavelet frame families $\cal{F}$ used for our numerical experiments  are constructed in \cite{first_paper} from the Dubuc-Deslauriers $2L$-point subdivision family, $L\in\mathbb{N}$. These semi-regular subdivision schemes
satisfy (i)-(iii) from Subsection 4.2. 
The frame families $\cal{F}$ fulfill assumptions (I)-(V) and are constructed as in \eqref{eq:WTF}. Moreover,  there exist $k_\ell(\mathbf{Q}),\;k_r(\mathbf{Q})\in\mathbb{Z}$, $k_\ell(\mathbf{Q})<k_r(\mathbf{Q})$ such that
\[\psi_{1,k}(x)\;=\;\left\{\begin{array}{cl}
\psi_{1,k_\ell(\mathbf{Q})}(\;x\;+\;h_\ell\;(k_\ell(\mathbf{Q})-k)/2\;) & \textrm{ for } k\leq k_\ell(\mathbf{Q}),\; k \equiv k_\ell(\mathbf{Q}) \mod 2 \\ \\
\psi_{1,k_\ell(\mathbf{Q})-1}(\;x\;+\;h_\ell\;(k_\ell(\mathbf{Q})-1-k)/2\;) & \textrm{ for } k< k_\ell(\mathbf{Q}),\; k \not\equiv k_\ell(\mathbf{Q}) \mod 2 \\ \\
\psi_{1,k_r(\mathbf{Q})}(\;x\;+\;h_r\;(k_r(\mathbf{Q})-k)/2\;) & \textrm{ for } k\geq k_r(\mathbf{Q}),\; k \equiv k_r(\mathbf{Q}) \mod 2 \\ \\
\psi_{1,k_r(\mathbf{Q})+1}(\;x\;+\;h_r\;(k_r(\mathbf{Q})+1-k)/2\;) & \textrm{ for } k>k_r(\mathbf{Q}),\; k \not\equiv k_r(\mathbf{Q}) \mod 2 \\ \\
\end{array}\right.,\quad x\in\mathbb{R}.\]
In particular, we only have $3+k_r(\mathbf{Q})-k_\ell(\mathbf{Q})$ different framelets, i.e. different refinement equations. This fact together with (iii), \eqref{eq:renorm} and \eqref{eq:WTF} guarantee that assumptions (II)-(V) in Section 1  are fulfilled.

The optimal H\"older-Zygmund exponents of the refinable functions in Subsections \ref{subsec:qB} and \ref{subsec:DD} are known. These examples are
used as a benchmark to test our theoretical results. In Subsection \ref{subsec:RBF}, we present a construction of new families of
semi-regular radial basis (RBF) refinable functions (generated by interpolatory RBF subdivision schemes) and determine their optimal H\"older-Zygmund exponents.  
The numerical computation were been done in MATLAB Version 2018a on a Windows 10 laptop.

\subsubsection{Quadratic B-spline scheme} \label{subsec:qB}
The quadratic B-spline scheme generates basic limit functions which are piecewise polynomials of degree two, supported between four consecutive knots of $\mathbf{t}$ in \eqref{eq:mesh}. The corresponding subdivision matrix $\mathbf{Z}$ is constructed to satisfy these conditions. There are $k_r(\mathbf{Z})-k_\ell(\mathbf{Z})-1=2$ irregular functions whose supports contain the point $\mathbf{t}(0)=0$ and it is well known that these functions are $\mathcal{C}^{2-\epsilon}(\mathbb{R})$, $\epsilon>0$, thus their optimal exponent is equal to $2$. 

\begin{figure}[h]
\begin{minipage}[c]{0.5\textwidth}
\[
\mathbf{Z}(-6:3,-3:0)\;=\;
\left[\begin{array}{cccc} 
1/4 \\ 
3/4 \\ 
3/4 & 1/4 & \\  
1/4 & 3/4 & \\  
& 5/6 & 1/6\\  
&1/3 & 2/3\\  
&& 3/4& 1/4\\  
&&1/4 & 3/4 \\ 
&&& 3/4 \\ 
&&&1/4
\end{array}\right]
\]
\end{minipage} 
\begin{minipage}[c]{0.5\textwidth}
\centering
\includegraphics[width=0.75\textwidth,height=0.21\textheight]{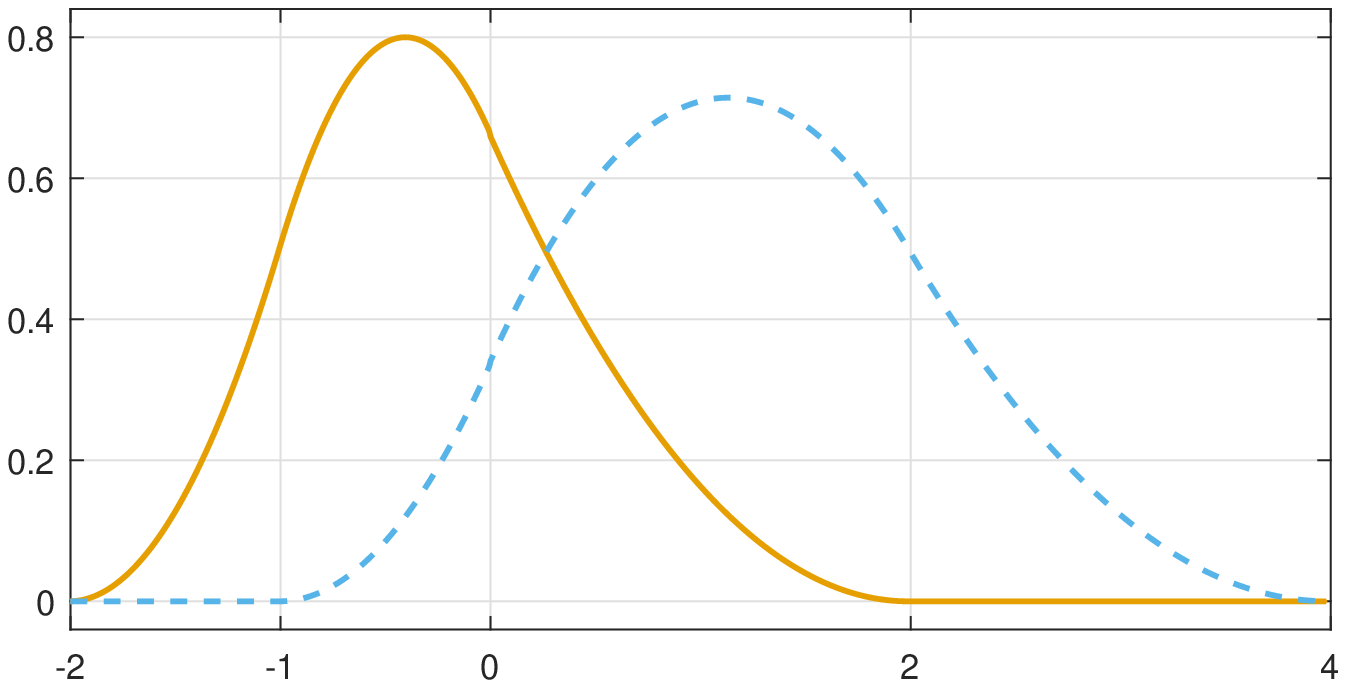}
\end{minipage}
$ $\\

 \begin{minipage}[c]{0.5\textwidth}
\begin{center}\begin{tabular}{|c|c|c|}
\hline
$n$ & $r_n(\zeta_{-2})$ & $r_n(\zeta_{-1})$ \\
\hline
   1&-0.9004    &0.8333\\
   2&-0.0362    &0.5235\\
   3 &0.6611    &0.9355\\
    4&1.0683    &1.2308\\
    5&1.3178    &1.4250\\
\hline
\end{tabular}\end{center}
\begin{center}\includegraphics[height=0.2\textheight,width=0.75\textwidth]{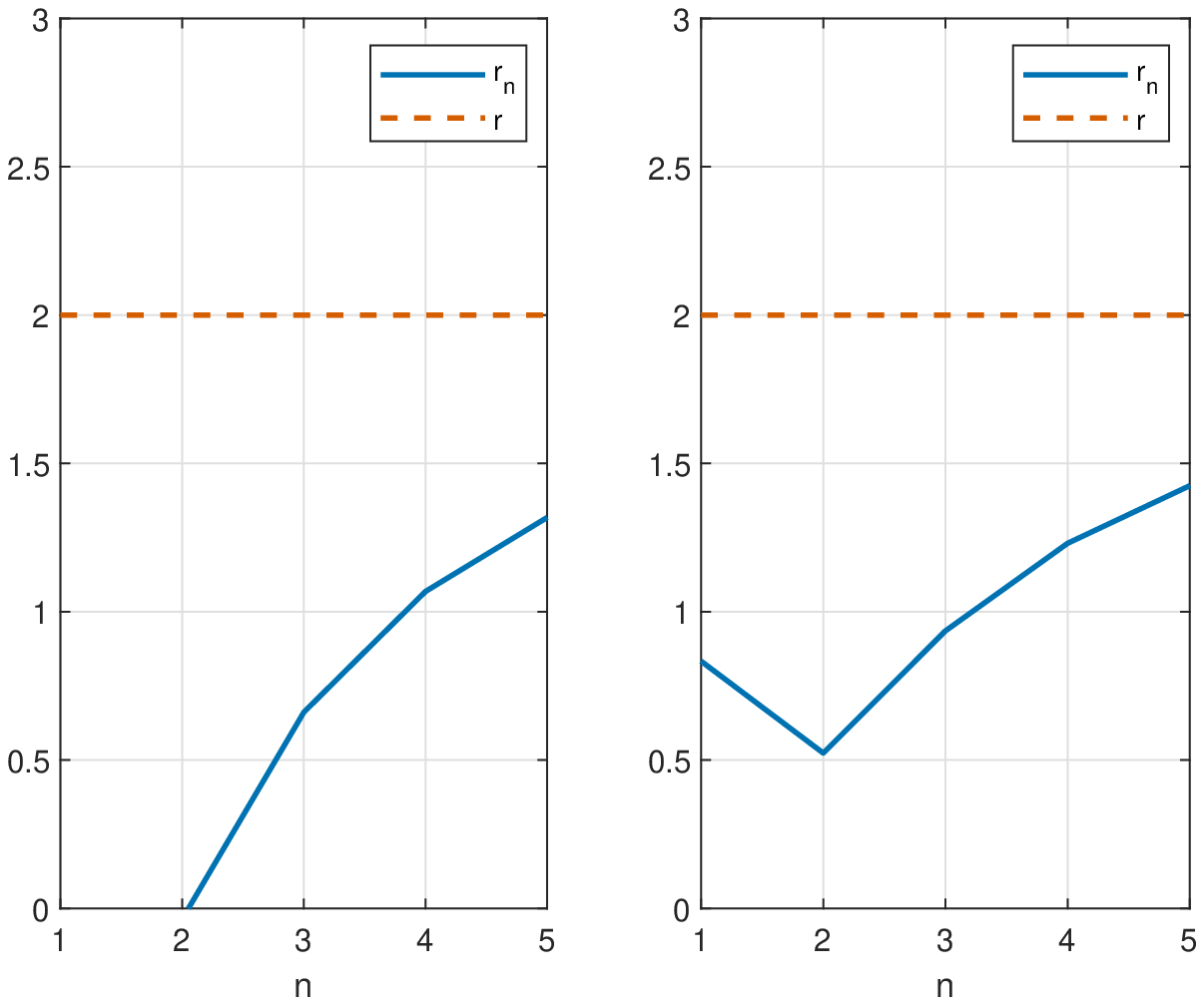} \end{center}
\end{minipage}
\begin{minipage}[c]{0.5\textwidth}
\begin{center}\begin{tabular}{|c|c|c|}
\hline
$n$ & $r^*_n(\zeta_{k_\ell(\mathbf{Z})+1})$ & $r^*_n(\zeta_{k_\ell(\mathbf{Z})+2})$ \\
\hline
   1&-0.9004   &0.8333\\
   2 &0.8280    &0.2136\\
   3 &2.0000    &2.0001\\
   4 &2.0000    &2.0000\\
   5 &2.0000    &2.0000\\
\hline
\end{tabular}\end{center}
\begin{center}\includegraphics[height=0.2\textheight,width=0.75\textwidth]{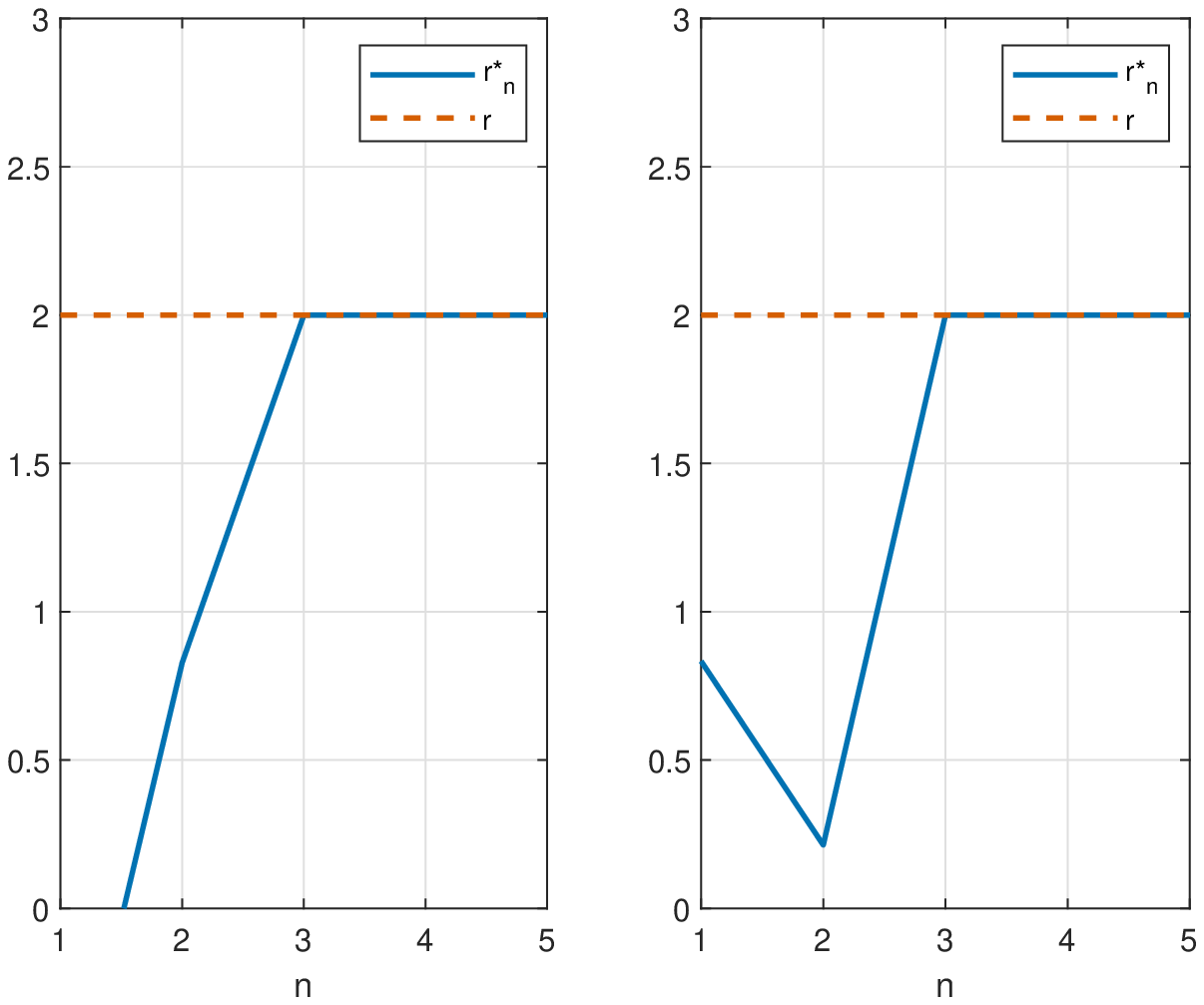}\end{center}
\end{minipage}
\caption{Semi-regular quadratic B-spline functions on the mesh $\bf{t}$ with $h_\ell=1$ and $h_r=2$ analyzed with the semi-regular 
Dubuc-Deslauriers $6$-point tight wavelet frame. Top row: part of the subdivision matrix $\mathbf{Z}$ that corresponds to the non-shift-invariant refinable functions around $0$ and
the graphs of these  functions $\zeta_{-2}$ and $\zeta_{-1}$. Middle row: estimates of the optimal H\"older-Zygmund exponents 
of $\zeta_{-2}$ and $\zeta_{-1}$ via linear regression (on the left) via the method in Proposition \ref{prop:my_test} (on the right).
Bottom row: graphs of the estimates of the H\"older-Zygmund exponents.}
\label{fig:B2_DD6}
\end{figure}

On Figure \ref{fig:B2_DD6} we give the estimates of the optimal H\"older-Zygmund exponents by both the linear regression method and by the
method in Proposition \ref{prop:my_test}.  For the analysis we used the semi-regular tight wavelet frame constructed from the limits of 
the semi-regular Dubuc-Deslauriers $6$-point subdivision scheme. This toy example already illustrates that the method proposed in Proposition \ref{prop:my_test} reaches the optimal exponent in few steps, while the linear regression method converges much slower.

\subsubsection{Dubuc-Deslauriers $4$-point scheme} \label{subsec:DD}
In \cite{MR1356994}, the subdivision matrix $\mathbf{Z}$ of the semi-regular  Dubuc-Deslauriers $4$-point scheme is obtained by requiring that $\mathbf{Z}$ is $2$-slanted, has at most $7$ non-zero entries in each column and $\mathbf{Z}(2i,k)=\delta_{i,k}$, $i,\;k\in\mathbb{Z}$, $\mathbf{Z}\;\mathbf{t}^\alpha\;=\;(\mathbf{t}/2)^\alpha$ for $\alpha\in\{0,\dots,3\}$ with $\mathbf{t}$ in \eqref{eq:mesh}. In this case, there are $5$ irregular (non-shift-invariant) refinable functions depicted on Figure \ref{fig:DD4_DD6}. Due to results in \cite{MR1687779}, it is well known that the optimal exponent of 
all these irregular functions is equal to $2$. Again, the method in Proposition \ref{prop:my_test} remarkably outperforms the linear regression method.

\begin{figure}[h]

\begin{minipage}[c]{0.5\textwidth}
{\scriptsize$\mathbf{Z}(-9:9,-3:3)\;=\;$
\[
\begin{bmatrix}
-1/16 \\
0\\ 
9/16&-1/16 \\
  1 &  0 \\
 9/16 &9/16 & -1/16 \\
  0 &  1&     0 \\
 -1/16 &9/16&  9/16& -1/16\\
   &  0&     1&     0 \\
 &-5/64&   5/8& 15/32& -1/64 \\
     &&     0&     1&     0 \\
     &&  -1/5&   3/4&   1/2& -1/20 \\
     &&&     0&     1&     0 \\
     &&& -1/16&  9/16&  9/16 & -1/16\\
     &&&&     0&     1& 0\\
     &&&& -1/16&  9/16 & 9/16\\
     &&&&&     0 & 1\\
     &&&&& -1/16 & 9/16\\
     &&&&&& 0 \\
     &&&&&& -1/16
\end{bmatrix}\]}
\end{minipage}
\begin{minipage}[c]{0.5\textwidth}
\centering
\includegraphics[width=0.9\textwidth]{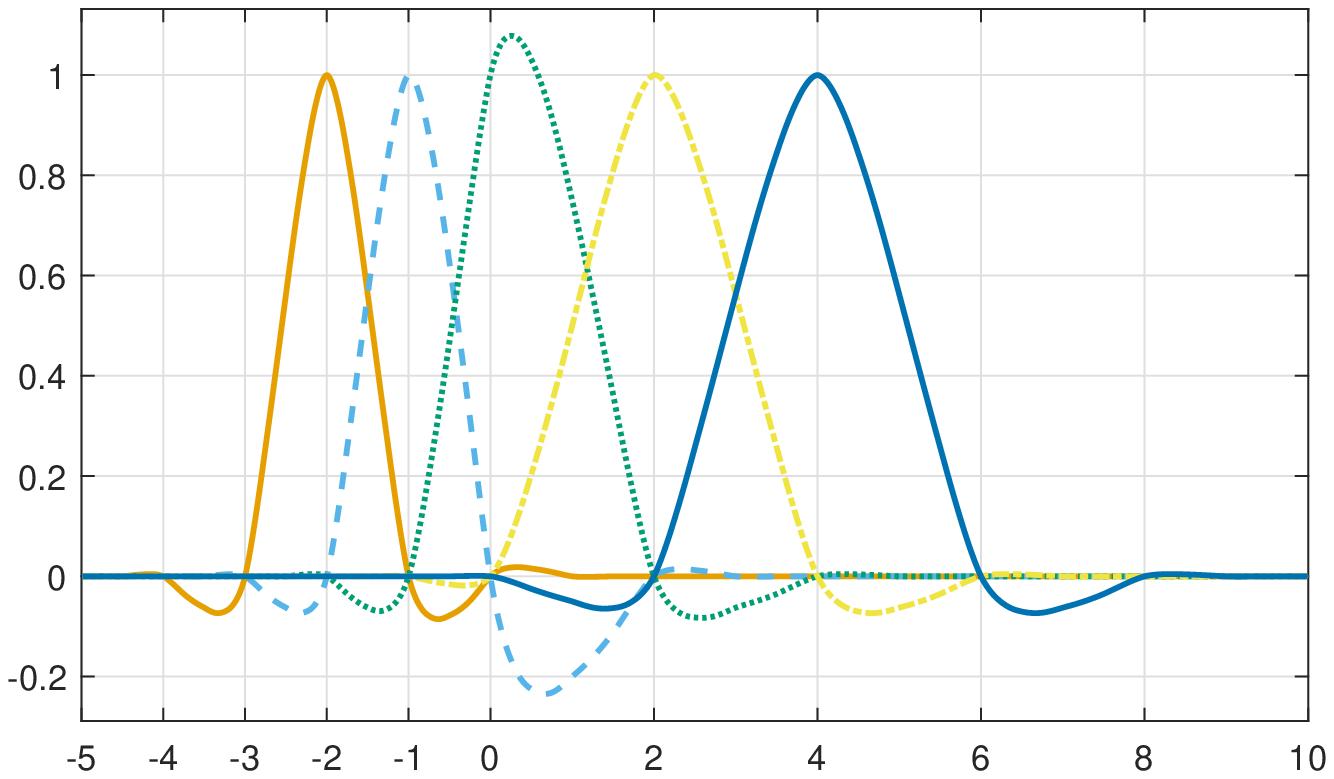}
\end{minipage}

\begin{minipage}[c]{0.5\textwidth}
{\footnotesize\begin{center}\begin{tabular}{|c|c|c|c|c|c|}
\hline
$n$ & $r_n(\zeta_{-2})$ & $r_n(\zeta_{-1})$ & $r_n(\zeta_{0})$ & $r_n(\zeta_{1})$ & $r_n(\zeta_{2})$\\
\hline
   1&-0.0589   &-0.6276   &-0.4561   &-0.6674    &1.4706\\
    2&1.5123    &0.4586    &0.6302    &0.8716    &1.8356\\
    3&1.8333    &1.0375    &1.1799    &1.5818    &2.2240\\
    4&1.9052    &1.3336    &1.4487    &1.8532    &2.3523\\
    5&1.9360    &1.5117    &1.6052    &1.9587    &2.3596\\
    6&1.9516    &1.6266    &1.7032    &2.0022    &2.3239\\
    7&1.9615    &1.7052    &1.7688    &2.0209    &2.2820\\
    8&1.9683    &1.7614    &1.8148    &2.0286    &2.2436\\
    9&1.9734    &1.8030    &1.8484    &2.0310    &2.2108\\
    10&1.9774    &1.8346    &1.8736    &2.0310    &2.1833\\
    11&1.9805    &1.8592    &1.8929    &2.0298    &2.1604\\
    12&1.9830    &1.8786    &1.9082    &2.0281    &2.1413\\
    13&1.9851    &1.8944    &1.9204    &2.0262    &2.1252\\
    14&1.9868    &1.9072    &1.9303    &2.0244    &2.1116\\
    15&1.9882    &1.9179    &1.9385    &2.0226    &2.1001\\
    16&1.9895    &1.9268    &1.9453    &2.0209    &2.0902\\
\hline
\end{tabular}\end{center}}
\end{minipage}
\begin{minipage}[c]{0.5\textwidth}
{\footnotesize\begin{center}\begin{tabular}{|c|c|c|c|c|c|}
\hline
$n$ & $r^*_n(\zeta_{-2})$ & $r^*_n(\zeta_{-1})$ & $r^*_n(\zeta_{0})$ & $r^*_n(\zeta_{1})$ & $r^*_n(\zeta_{2})$\\
\hline
   1&-0.0589   &-0.6276   &-0.4561   &-0.6674    &1.4706\\
    2&3.0835    &1.5448    &1.7165    &2.4105    &2.2007\\
    3&2.0585    &2.0261    &2.1006    &2.7262    &3.0086\\
    4&1.8721    &1.9392    &1.9742    &2.2284    &2.4769\\
    5&1.9764    &1.9883    &2.0064    &2.0489    &2.1474\\
    6&1.9792    &1.9897    &1.9984    &2.0131    &2.0013\\
    7&1.9920    &1.9960    &2.0004    &2.0096    &1.9997\\
    8&1.9954    &1.9977    &1.9999    &2.0041    &2.0001\\
    9&1.9979    &1.9989    &2.0000    &2.0022    &2.0000\\
    10&1.9989    &1.9995    &2.0000    &2.0011    &2.0000\\
    11&1.9995    &1.9997    &2.0000    &2.0005    &2.0000\\
    12&1.9997    &1.9999    &2.0000    &2.0003    &2.0000\\
    13&1.9999    &1.9999    &2.0000    &2.0001    &2.0000\\
    14&1.9999    &2.0000    &2.0000    &2.0001    &2.0000\\
    15&2.0000    &2.0000    &2.0000    &2.0000    &2.0000\\
    16&2.0000    &2.0000    &2.0000    &2.0000    &2.0000\\
\hline
\end{tabular}\end{center}}
\end{minipage}

$ $\\

\begin{minipage}[c]{0.5\textwidth}
\begin{center}\includegraphics[width=0.9\textwidth]{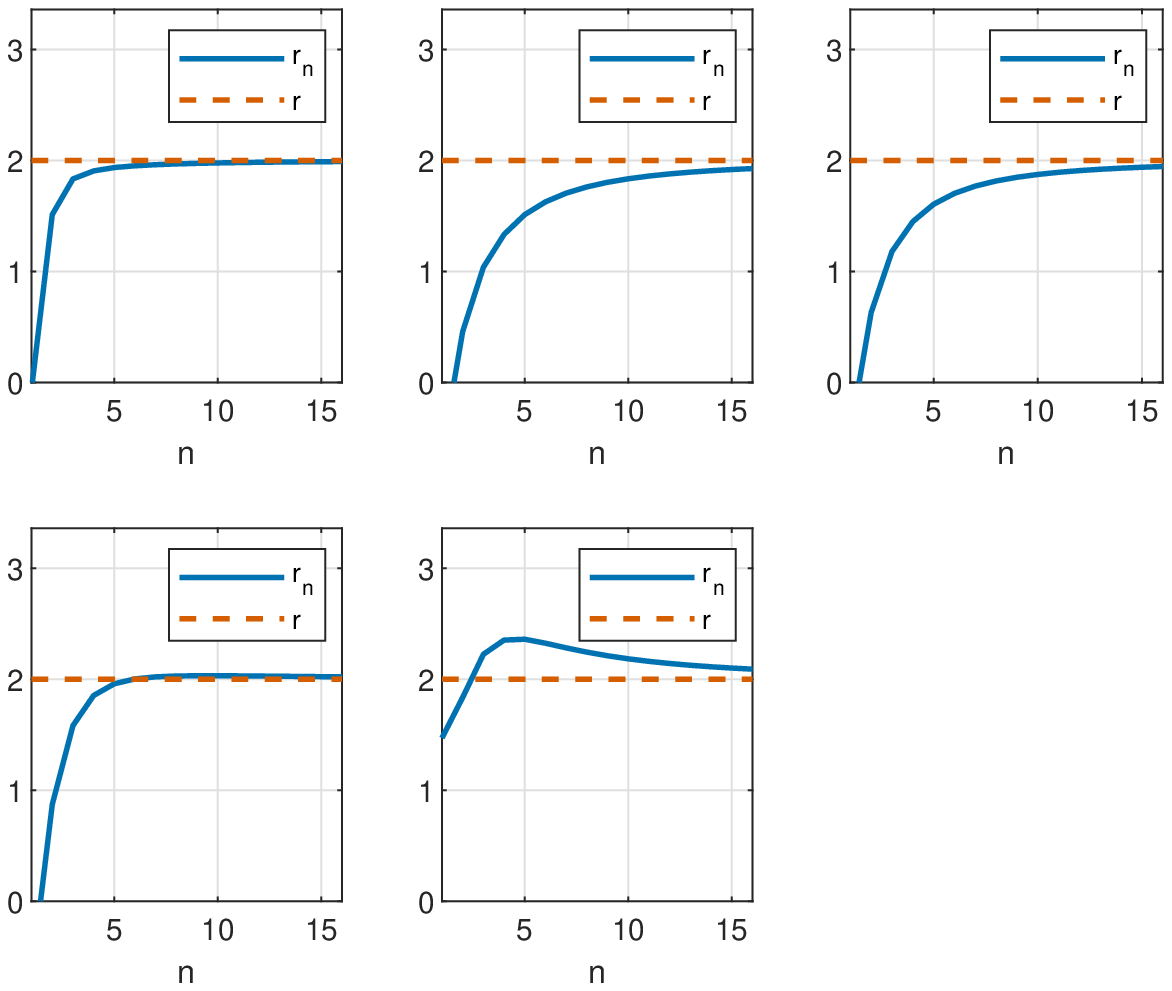} \end{center}
\end{minipage}
\begin{minipage}[c]{0.5\textwidth}
\begin{center}\includegraphics[width=0.9\textwidth]{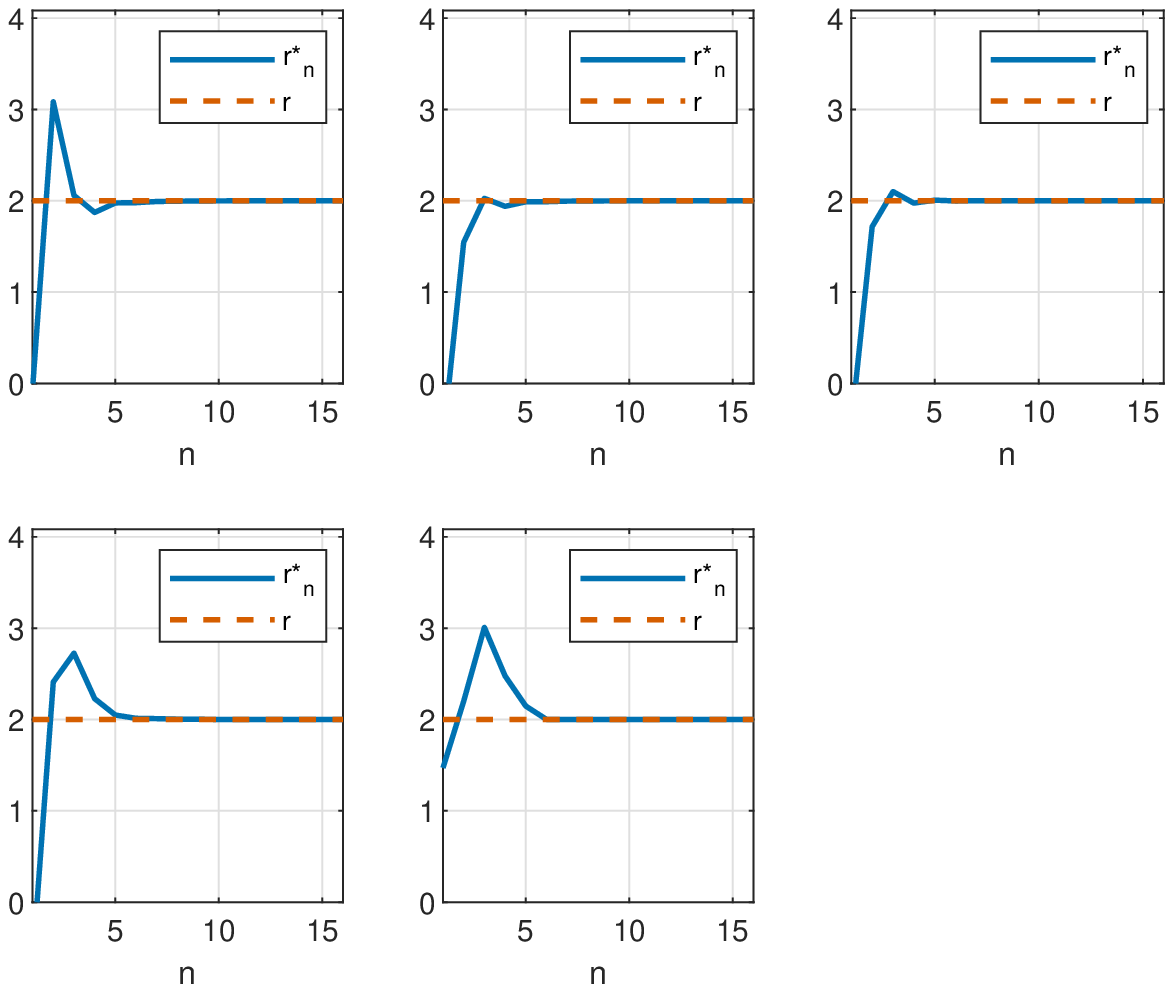}\end{center}
 \end{minipage}

\caption{Semi-regular Dubuc-Deslauriers $4$-point limit functions on the mesh $\bf{t}$ with $h_\ell=1$ and $h_r=2$ analyzed via the semi-regular Dubuc-Deslauriers $6$-point tight wavelet frame.  Top row: part of the subdivision matrix $\mathbf{Z}$ that corresponds to the non-shift-invariant refinable functions around $0$ and the graphs of these functions $\zeta_{-2},\dots,\zeta_{2}$. Middle row: estimates of the optimal H\"older-Zygmund exponents 
of $\zeta_{-2},\dots,\zeta_{2}$ via linear regression (on the left) via the method in Proposition \ref{prop:my_test} (on the right).
Bottom row: graphs of the estimates of the H\"older-Zygmund exponents.}
\label{fig:DD4_DD6}
\end{figure}

\subsubsection{Radial basis functions based interpolatory schemes} \label{subsec:RBF} 

Using  techniques similar to \cite{first_paper, Warren:2001:SMG:580358}, we extend the subdivision schemes \cite{MR2231695, MR2578850} based on radial basis functions to the semi-regular setting. Let $L\in\mathbb{N}$. We require that the subdivision matrix $\mathbf{Z}$ satisfies
 $\mathbf{Z}(2i,k)=\delta_{i,k}$ for $i, k\in\mathbb{Z}$. To determine the other entries of the $2$-slanted matrix $\mathbf{Z}$ whose columns 
are centered at $\mathbf{Z}(2k,k)$, $k\in\mathbb{Z}$ and have support length at most $4L-1$, we proceed as follows. We first choose a 
radial basis function $g(x)=g(|x|)$, $x \in \mathbb{R}$, which is conditionally positive definite of order $\eta\in\mathbb{N}$, i.e., for
every set of pairwise distinct points $\{x_i\}_{i=1}^{N}\subset\mathbb{R}$ and coefficients $\{c_i\}_{i=1}^{N}\subset\mathbb{R}$, $N\in\mathbb{N}$, there
exists a polynomial $\pi$ of degree at most $\eta-1$ such that
\[
\sum_{i=1}^{N}\;c_i\;\pi(x_i)\;=\;0
\]
and the function $g$ satisfies
\[
 \sum_{i=1}^{N}\;\sum_{k=1}^{N}\;c_i\;c_k\;g(x_i-x_k)\;\geq\;0.
\]
The next step is to choose the order $m\in\{ \eta,\dots,2L\}$ of polynomial reproduction and, for every set of $2L$ consecutive points $\mathbf{t}(k-L+1),\;\dots,\;\mathbf{t}(k+L)$, $k \in \mathbb{Z}$, of the mesh $\mathbf{t}$ in \eqref{eq:mesh}, solve the linear system of equations
\begin{equation} \label{eq:RBFs_system}
\begin{bmatrix}
\mathbf{A} & \mathbf{B} \\ \\
\mathbf{B}^T & \mathbf{0}
\end{bmatrix}\;\begin{bmatrix}
\mathbf{u} \\ \\
\mathbf{v}
\end{bmatrix}\;=\;\begin{bmatrix}
\mathbf{r} \\ \\
\mathbf{s}
\end{bmatrix}\end{equation}
with
$$ \{\;\mathbf{A}(i,j)\;=\;g(\mathbf{t}(k-L+i)-\mathbf{t}(k-L+j))\;\}_{i,j=1,\dots,2L}, \quad 
 \{\;\mathbf{B}(i,j)\;=\;\mathbf{t}(k-L+i)^{j-1}\;\}_{i=1,\dots,2L,\;j=1,\dots,m},
$$
$$
 \{\;\mathbf{r}(i)\;=\;g(x_k-\mathbf{t}(k-L+i))\;\}_{i=1,\dots,2L},\quad\{\;\mathbf{s}(j)\;=\;x_k^{j-1}\;\}_{j=1,\dots,m}
\quad  \hbox{and} $$
$$ 
 x_k \;=\; (\;\mathbf{t}(k)+\mathbf{t}(k+1)\;)/2.
$$
Lastly, the vector $\mathbf{u}$ contains the entries of the $2k+1$-th row of $\mathbf{Z}$ associated to the columns $k-L+1$ to $k+L$. 

\begin{rem} \label{rem:RBF}

$(i)$ Determining the rows of $\mathbf{Z}$ by solving the linear systems  \eqref{eq:RBFs_system} for $k \in \mathbb{Z}$ guarantees the polynomial reproduction of degree 
at most $m-1$. Indeed, the condition $\mathbf{B}^T\;\mathbf{u}\;=\;\mathbf{s}$ forces $\mathbf{Z}$ to map samples over $\mathbf{t}$ of a polynomial of degree at most $m-1$ onto sample over the finer knots $\mathbf{t}/2$ of the same polynomial. 

\noindent $(ii)$ If $m=2L$, the system of equations $\mathbf{B}^T\;\mathbf{u}\;=\;\mathbf{s}$ coincides with the one defining the Dubuc-Deslauriers $2L$-point scheme in \cite{first_paper}. In this case,
the system  $\mathbf{B}^T\;\mathbf{u}\;=\;\mathbf{s}$ has a unique solution, which makes the presence of $\mathbf{A}$, i.e. of the radial basis function $g$
obsolete.

\noindent $(iii)$ If $m<2L$, the structure of the irregular limit functions around $0$ reflect the transition (blending) between the two (one on the left and one on the right of $0$) subdivision schemes of
different regularity, see Example \ref{ex:Buhmann}. This depends on the properties of the chosen underlying radial basic function $g$. 
For example, the blending produces no visible effect if $g$ is homogeneous, i.e. 
$g(\lambda x )=|\lambda| g(x)$, $\lambda\in\mathbb{R}$. In this case, for $\lambda>0$, the linear system of equations
\begin{equation} \label{eq:RBFs_system2}
\begin{bmatrix}
\lambda\mathbf{I} & \mathbf{0} \\ \\
\mathbf{0} & \mathbf{L}
\end{bmatrix}\;\begin{bmatrix}
\mathbf{A} & \mathbf{B} \\ \\
\mathbf{B}^T & \mathbf{0}
\end{bmatrix}\;\begin{bmatrix}
\mathbf{I} & \mathbf{0} \\ \\
\mathbf{0} & \mathbf{L}/\lambda
\end{bmatrix}\;\begin{bmatrix}
\mathbf{I} & \mathbf{0} \\ \\
\mathbf{0} & \lambda\mathbf{L}^{-1}
\end{bmatrix}\;\begin{bmatrix}
\mathbf{u} \\ \\
\mathbf{v}
\end{bmatrix}\;=\;\begin{bmatrix}
\lambda\mathbf{I} & \mathbf{0} \\ \\
\mathbf{0} & \mathbf{L}
\end{bmatrix}\;\begin{bmatrix}
\mathbf{r} \\ \\
\mathbf{s}
\end{bmatrix}
\end{equation}
with the identity matrix $\mathbf{I}$  and $\mathbf{L}=\diag([\lambda^{j-1}]_{j=1,\dots,m})$ is equivalent to the system in \eqref{eq:RBFs_system} 
for the mesh $\lambda\mathbf{t}$. The structure of the linear system in \eqref{eq:RBFs_system2} implies that $\mathbf{u}$ is the same 
as the one determined by \eqref{eq:RBFs_system}. 

\noindent $(iv)$ The argument in $(iii)$ with $\mathbf{L}=\mathbf{I}$ shows that the subdivision matrix obtained this way does not depend on the normalization of the radial basis function $g$, i.e. all functions $\lambda g$, $\lambda>0$,  lead to the same subdivision scheme. 
\end{rem}

\begin{exmp} \label{ex:Buhmann} We consider the radial basis function introduced by M.~Buhmann in  \cite{MR1883626}
\begin{equation} \label{eq:BU}
	g(x)\;=\;\left\{\begin{array}{rl}
		12x^4\log|x|\;-\;21x^4\;+\;32|x|^3\;-\;12x^2\;+\;1,& \textrm{ if } |x|<1, \\ \\
		0, & \textrm{ otherwise},
	\end{array}\right.
\end{equation}
and choose $L=2$ and $m=1$. The resulting irregular functions $\zeta_{-2},\dots,\zeta_{2}$ are shown on Figure \ref{fig:GA_DD4}. The structure
of $\zeta_{0}$ illustrates the blending effect (described in Remark \ref{rem:RBF} part $(iii)$) of two different subdivision schemes meeting at $0$. 
Figure \ref{fig:GA_DD4} also presents the estimates of the optimal H\"older-Zygmund exponents of $\zeta_{-2},\dots,\zeta_{2}$.
These exponents are determined using the tight wavelet frame \cite{first_paper} based on the Dubuc-Deslauriers $4$-point subdivision scheme. 
We again observe the phenomenon that the method in Proposition \ref{prop:my_test} converges faster than the linear regression. 
\end{exmp}

\begin{figure}[h]

\begin{minipage}[c]{0.5\textwidth}
{\scriptsize$\mathbf{Z}(-9:9,-3:3)\;=\;$
\[\left[\begin{array}{ccccccc}
	0.1662 & & & & & \\ 
	0& & & & & \\
	0.3338& 0.1662 &  &  &  & &\\ 
	1&0 &  &  &  &  &\\ 
	0.3338& 0.3338 & 0.1662 &  &  &  &\\ 
	0&1 & 0 &  &  &  &\\ 
	0.1662& 0.3338 & 0.3338 & 0.1662 &  &  &\\ 
	&0 & 1 & 0 &  &  &\\ 
	&0.1662 & 0.3338 & 0.3338 &0.1662 &  &\\ 
	& & 0 & 1 & 0 &  &\\ 
	& & 0.2500 & 0.2500 & 0.2500 & 0.2500 &\\ 
	& &  & 0 & 1 & 0 &\\ 
	& &  & 0.2500 & 0.2500 & 0.2500 & 0.2500\\ 
	& &  &  & 0 & 1 & 0\\ 
	& &  &  & 0.2500 & 0.2500 & 0.2500\\ 
	& &  &  &  & 0 & 1\\ 
	& &  &  &  & 0.2500 & 0.2500\\
	& & & & & & 0\\
	& & & & & & 0.2500\\
\end{array}\right].\]}

\end{minipage}
\begin{minipage}[c]{0.5\textwidth}
\centering
\includegraphics[width=0.9\textwidth,height=0.2\textheight]{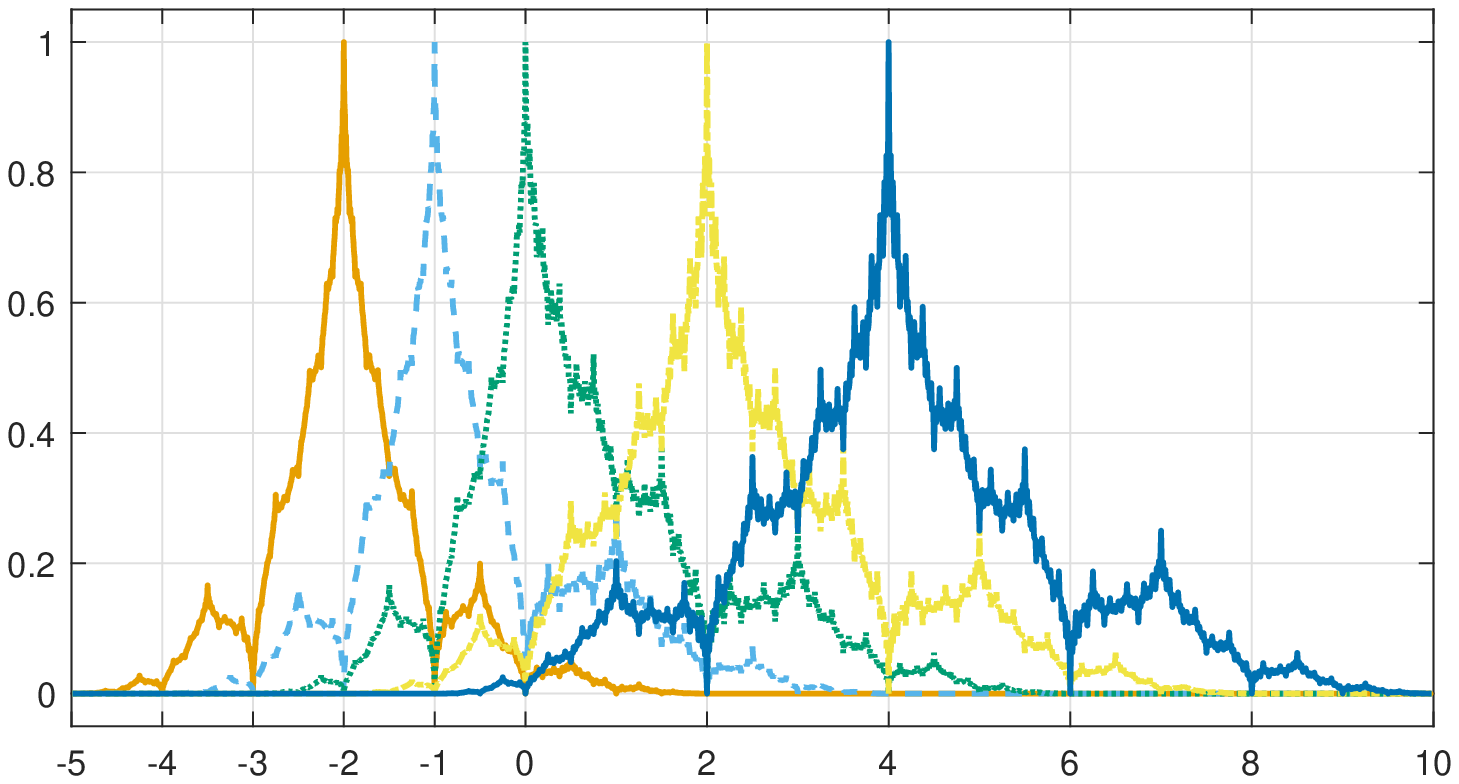} 
\end{minipage}

\begin{minipage}[c]{0.5\textwidth}
{\footnotesize\begin{center}\begin{tabular}{|c|c|c|c|c|c|}
\hline
$n$ & $r_n(\zeta_{-2})$ & $r_n(\zeta_{-1})$ & $r_n(\zeta_{0})$ & $r_n(\zeta_{1})$ & $r_n(\zeta_{2})$\\
\hline
1&    2.3107&    0.0955&   -0.3137&    1.9167&    1.1751\\
2&    1.3713&    0.7568&    0.1025&    1.2398&    0.7305\\
3&    1.0062&    0.5966&    0.2212&    0.9391&    0.5609\\
4&    0.8229&    0.5853&    0.2821&    0.7386&    0.4754\\
5&    0.7185&    0.5437&    0.3146&    0.6290&    0.4269\\
6&    0.6532&    0.5226&    0.3342&    0.5552&    0.3965\\
7&    0.6097&    0.5014&    0.3466&    0.5037&    0.3763\\
8&    0.5793&    0.4860&    0.3550&    0.4664&    0.3622\\
9&    0.5571&    0.4726&    0.3608&    0.4387&    0.3519\\
10&    0.5405&    0.4618&    0.3650&    0.4175&    0.3442\\
11&    0.5278&    0.4525&    0.3681&    0.4010&    0.3383\\
12&    0.5177&    0.4448&    0.3705&    0.3879&    0.3336\\
13&    0.5097&    0.4382&    0.3723&    0.3773&    0.3299\\
14&    0.5032&    0.4325&    0.3737&    0.3686&    0.3269\\
15&    0.4978&    0.4276&    0.3749&    0.3614&    0.3244\\
16&    0.4934&    0.4234&    0.3758&    0.3554&    0.3223\\
\hline
\end{tabular}\end{center}}
\end{minipage}
\begin{minipage}[c]{0.5\textwidth}
{\footnotesize\begin{center}\begin{tabular}{|c|c|c|c|c|c|}
\hline
$n$ & $r^*_n(\zeta_{-2})$ & $r^*_n(\zeta_{-1})$ & $r^*_n(\zeta_{0})$ & $r^*_n(\zeta_{1})$ & $r^*_n(\zeta_{2})$\\
\hline
1&    2.3107&    0.0955&   -0.3137&    1.9167&    1.1751\\
2&    0.4319&    1.4181&    0.5186&    0.5630&    0.2859\\
3&    0.4673&    0.0025&    0.3595&    0.4631&    0.3133\\
4&    0.4549&    0.7000&    0.4071&    0.2370&    0.3029\\
5&    0.4585&    0.3168&    0.3879&    0.3727&    0.3069\\
6&    0.4573&    0.4855&    0.3888&    0.3053&    0.3053\\
7&    0.4577&    0.3743&    0.3860&    0.3059&    0.3059\\
8&    0.4576&    0.4187&    0.3846&    0.3057&    0.3057\\
9&    0.4576&    0.3832&    0.3839&    0.3058&    0.3058\\
10&    0.4576&    0.3957&    0.3831&    0.3057&    0.3057\\
11&    0.4576&    0.3838&    0.3829&    0.3058&    0.3058\\
12&   0.4576&    0.3872&    0.3825&    0.3058&    0.3058\\
13&    0.4576&    0.3832&    0.3824&    0.3058&    0.3058\\
14&    0.4576&    0.3841&    0.3823&    0.3058&    0.3058\\
15&    0.4576&    0.3827&    0.3822&    0.3058&    0.3058\\
16&    0.4576&    0.3829&    0.3822&    0.3058&    0.3058\\
\hline
\end{tabular}\end{center}}
\end{minipage}

$ $\\

\begin{minipage}[c]{0.5\textwidth}
\begin{center}
\includegraphics[width=0.9\textwidth]{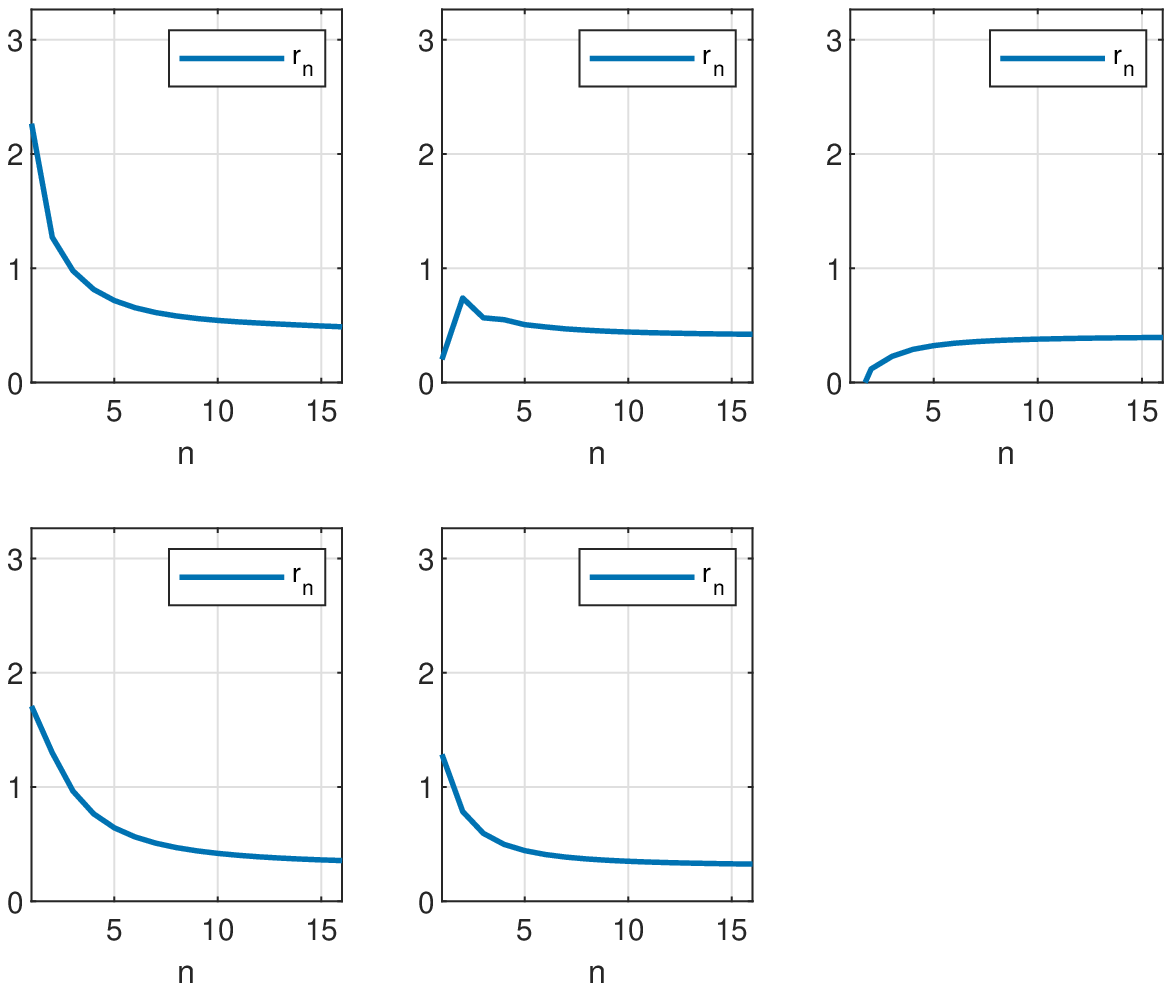}
\end{center}
\end{minipage}
\begin{minipage}[c]{0.5\textwidth}
\begin{center}
\includegraphics[width=0.9\textwidth]{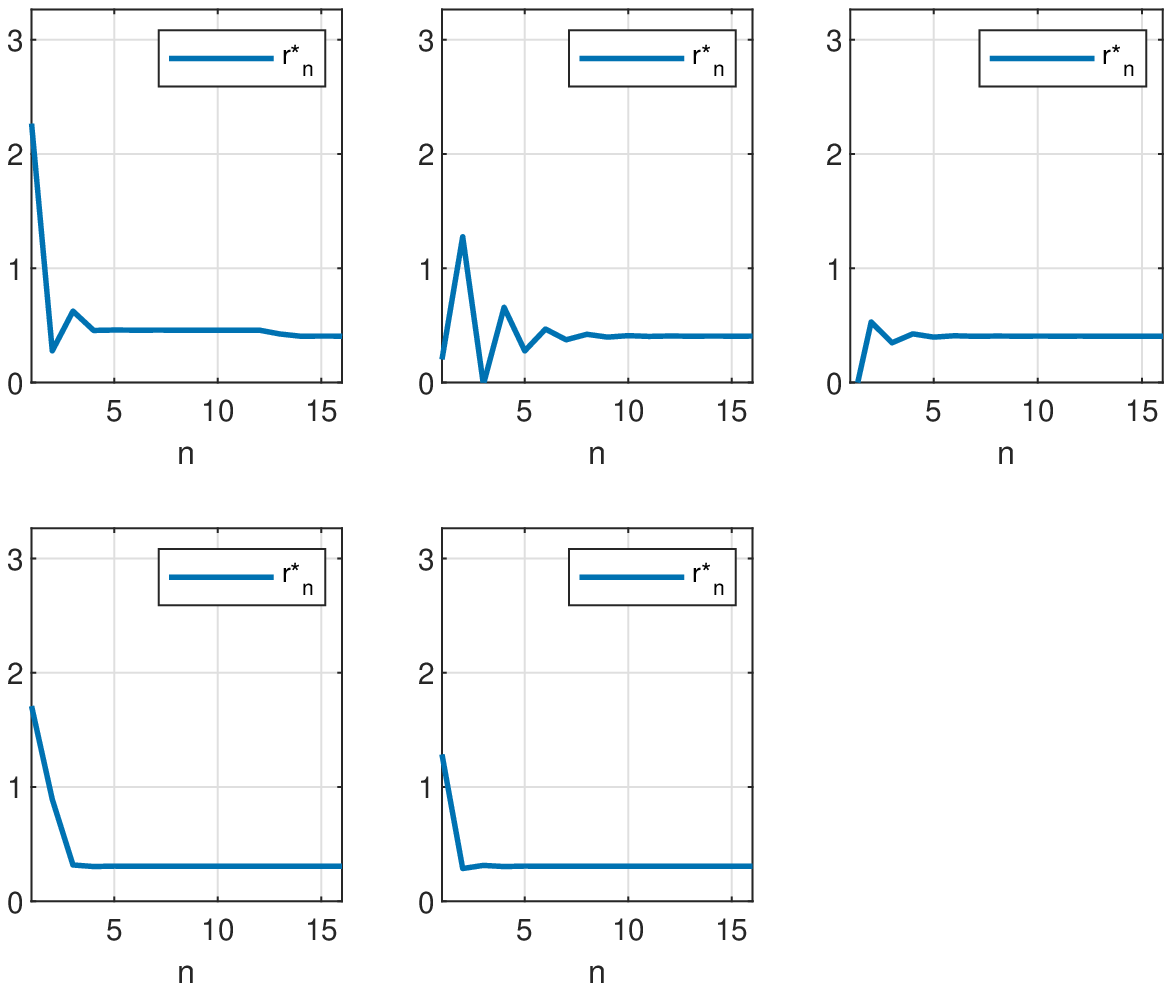} 
\end{center}
\end{minipage}

\caption{Semi-regular interpolatory subdivision scheme based on $g$ in \eqref{eq:BU}, $L=2$, $m=1$, on the mesh $\bf{t}$ with $h_\ell=1$ and $h_r=2$ analyzed via the semi-regular Dubuc-Deslauriers $4$-point tight wavelet frame.  Top row: part of the subdivision matrix $\mathbf{Z}$ that corresponds to the non-shift-invariant refinable functions around $0$ and the graphs of these functions $\zeta_{-2},\dots,\zeta_{2}$. Middle row: estimates of the optimal H\"older-Zygmund exponents 
of $\zeta_{-2},\dots,\zeta_{2}$ via linear regression (on the left) via the method in Proposition \ref{prop:my_test} (on the right).
Bottom row: graphs of the estimates of the H\"older-Zygmund exponents.}
\label{fig:GA_DD4}
\end{figure}

\begin{exmp} Another radial basis function that we consider is the polyharmonic function $g(x)\;=\;|x|$, $x \in \mathbb{R}$. 
The corresponding irregular part of the interpolatory subdivision matrix  $\mathbf{Z}$ is determined for $L=2$ and $m=3$, see Figure~\ref{fig:PS_DD6}.
Note that the regular part of the subdivision matrix $\mathbf{Z}$ (see the first and the last columns corresponding
to the regular parts of the mesh)  coincides with the subdivision matrix of the regular Dubuc-Deslauriers $4$-point scheme. Due to
the observation in Remark \ref{rem:RBF} part $(iii)$, the absence of the blending effect is due to our choice of a homogeneous function $g$. 
We would like to emphasise that the resulting subdivision scheme around $0$ is not the semi-regular Dubuc-Deslauriers $4$-point scheme, compare with Figure \ref{fig:DD4_DD6}. Indeed, the polynomial reproduction around $0$ is of one degree lower. We also lose regularity (the Dubuc-Deslauriers $4$-point scheme is $C^{2-\epsilon}$, $\epsilon>0$) but overall the irregular limit functions on Figure~\ref{fig:PS_DD6} have a more uniform behavior than those on Figure \ref{fig:DD4_DD6}. We again observe that the method in Proposition \ref{prop:my_test} yields better estimates for the optimal H\"older-Zygmund exponent, see tables on Figure~\ref{fig:PS_DD6}.
\end{exmp}

\begin{figure}[h]

\begin{minipage}[c]{0.5\textwidth}
{\scriptsize$\mathbf{Z}(-9:9,-3:3)\;=\;$
\[\left[\begin{array}{ccccccc}
-1/16 \\
0\\ 
9/16&-1/16 \\
  1 &  0 \\
 9/16 &9/16 & -1/16 \\
  0 &  1&     0 \\
 -1/16 &9/16&  9/16& -1/16\\
   &  0&     1&     0 \\
 &-1/24& 19/36& 13/24& -1/36& \\
     &&     0&     1&     0 \\
     &&  -1/9&  7/12& 11/18& -1/12 \\
     &&&     0&     1&     0 \\
     &&& -1/16&  9/16&  9/16 & -1/16\\
     &&&&     0&     1& 0\\
     &&&& -1/16&  9/16 & 9/16\\
     &&&&&     0 & 1\\
     &&&&& -1/16 & 9/16\\
     &&&&&& 0 \\
     &&&&&& -1/16
\end{array}\right].\]}

\end{minipage}
\begin{minipage}[c]{0.5\textwidth}
\centering
\includegraphics[width=0.9\textwidth,height=0.2\textheight]{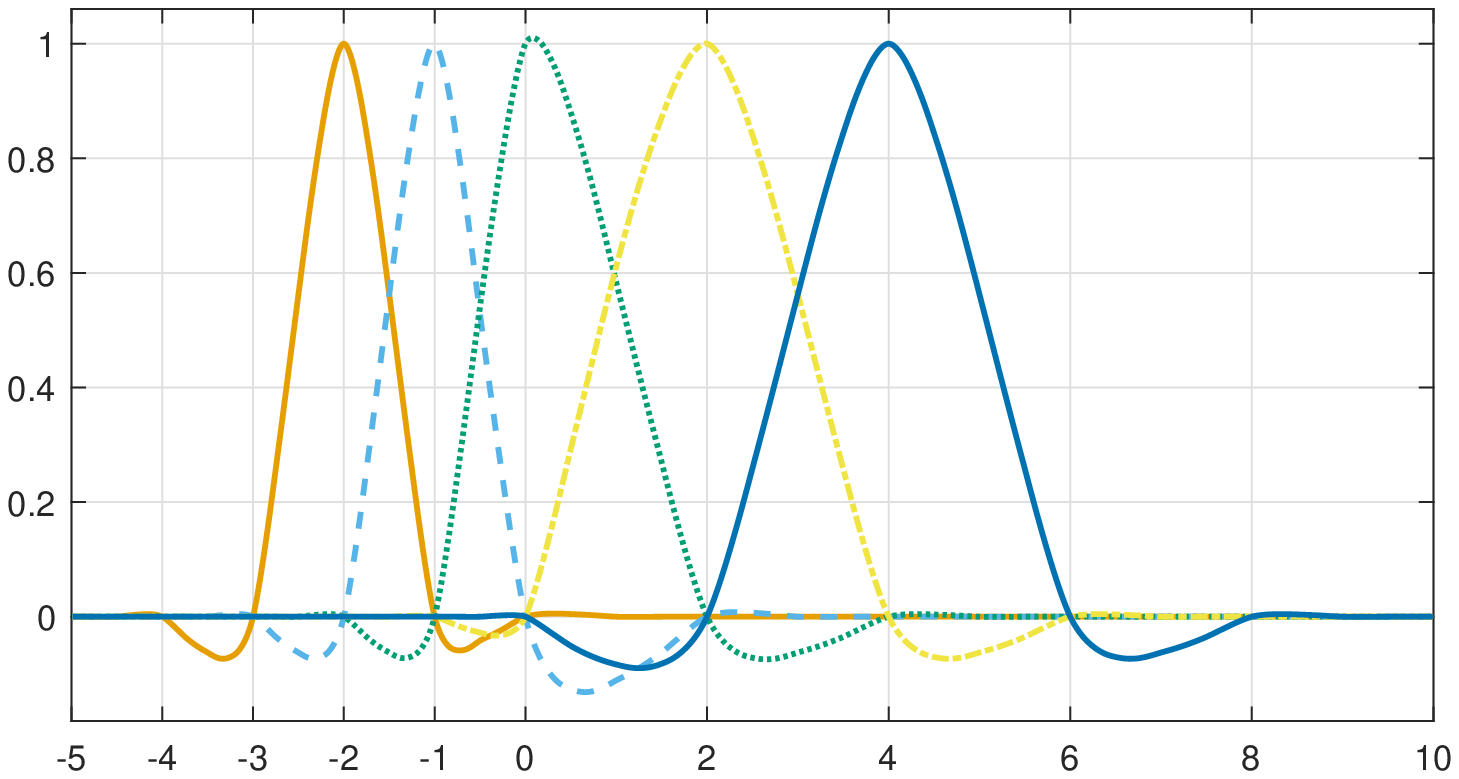} 
\end{minipage}

\begin{minipage}[c]{0.5\textwidth}
{\footnotesize\begin{center}\begin{tabular}{|c|c|c|c|c|c|}
\hline
$n$ & $r_n(\zeta_{-2})$ & $r_n(\zeta_{-1})$ & $r_n(\zeta_{0})$ & $r_n(\zeta_{1})$ & $r_n(\zeta_{2})$\\
\hline
1&   -0.0163&   -0.5174&   -0.3984&   -0.5637&    1.4525\\
2&    1.8110&    0.6301&    0.2509&    0.5646&    1.5324\\
3&    2.3259&    0.9516&    0.7115&    1.0643&    1.6848\\
4&    2.1635&    1.1421&    1.0006&    1.3048&    1.7340\\
5&    2.0346&    1.2779&    1.1895&    1.4400&    1.7552\\
6&    1.9538&    1.3766&    1.3178&    1.5230&    1.7640\\
7&    1.9029&    1.4495&    1.4081&    1.5777&    1.7677\\
8&    1.8695&    1.5043&    1.4740&    1.6155&    1.7692\\
9&    1.8466&    1.5463&    1.5232&    1.6429&    1.7696\\
10&    1.8303&    1.5791&    1.5610&    1.6633&    1.7695\\
11&    1.8183&    1.6050&    1.5906&    1.6789&    1.7693\\
12&    1.8092&    1.6260&    1.6141&    1.6911&    1.7689\\
13&    1.8022&    1.6430&    1.6332&    1.7008&    1.7685\\
14&    1.7967&    1.6571&    1.6488&    1.7087&    1.7681\\
15&    1.7922&    1.6689&    1.6618&    1.7152&    1.7677\\
16&    1.7886&    1.6788&    1.6727&    1.7206&    1.7674\\
\hline
\end{tabular}\end{center}}
\end{minipage}
\begin{minipage}[c]{0.5\textwidth}
{\footnotesize\begin{center}\begin{tabular}{|c|c|c|c|c|c|}
\hline
$n$ & $r^*_n(\zeta_{-2})$ & $r^*_n(\zeta_{-1})$ & $r^*_n(\zeta_{0})$ & $r^*_n(\zeta_{1})$ & $r^*_n(\zeta_{2})$\\
\hline
1&   -0.0163&   -0.5174&   -0.3984&   -0.5637&    1.4525\\
2&    3.6383&    1.7777&    0.9002&    1.6929&    1.6124\\
3&    2.9182&    1.3190&    1.5699&    1.8541&    2.0135\\
4&    0.9991&    1.5830&    1.6965&    1.7672&    1.7787\\
5&    1.5859&    1.7110&    1.7445&    1.7700&    1.7838\\
6&    1.7105&    1.7467&    1.7568&    1.7642&    1.7665\\
7&    1.7468&    1.7580&    1.7612&    1.7638&    1.7649\\
8&    1.7580&    1.7615&    1.7625&    1.7633&    1.7636\\
9&    1.7615&    1.7626&    1.7630&    1.7632&    1.7633\\
10&    1.7626&    1.7630&    1.7631&    1.7632&    1.7632\\
11&    1.7630&    1.7631&    1.7632&    1.7632&    1.7632\\
12&    1.7631&    1.7632&    1.7632&    1.7632&    1.7632\\
13&    1.7632&    1.7632&    1.7632&    1.7632&    1.7632\\
14&    1.7632&    1.7632&    1.7632&    1.7632&    1.7632\\
15&    1.7632&    1.7632&    1.7632&    1.7632&    1.7632\\
16&    1.7632&    1.7632&    1.7632&    1.7632&    1.7632\\
\hline
\end{tabular}\end{center}}
\end{minipage}

$ $\\

\begin{minipage}[c]{0.5\textwidth}
\begin{center}
\includegraphics[width=0.9\textwidth]{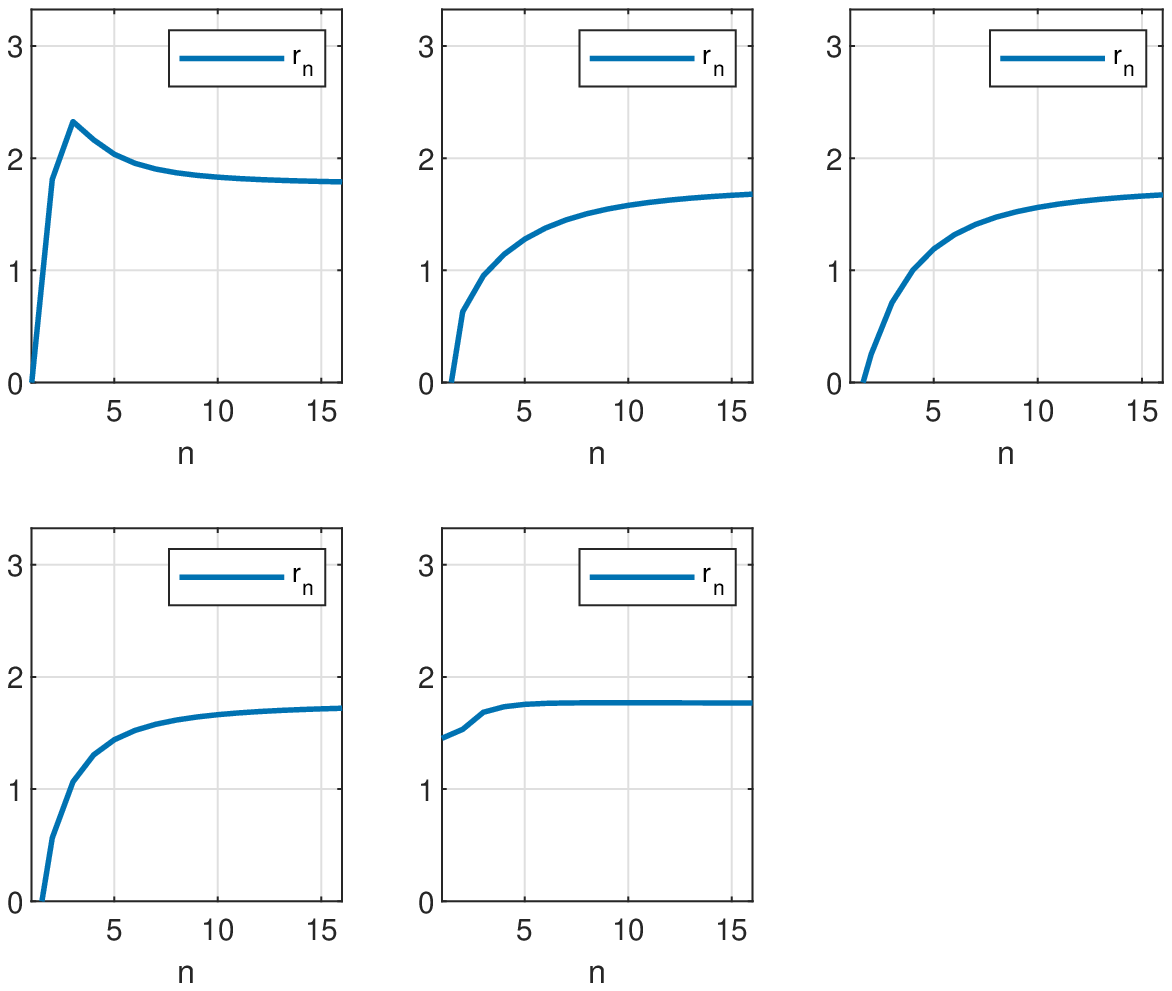} 
\end{center}
\end{minipage}
\begin{minipage}[c]{0.5\textwidth}
\begin{center}
\includegraphics[width=0.9\textwidth]{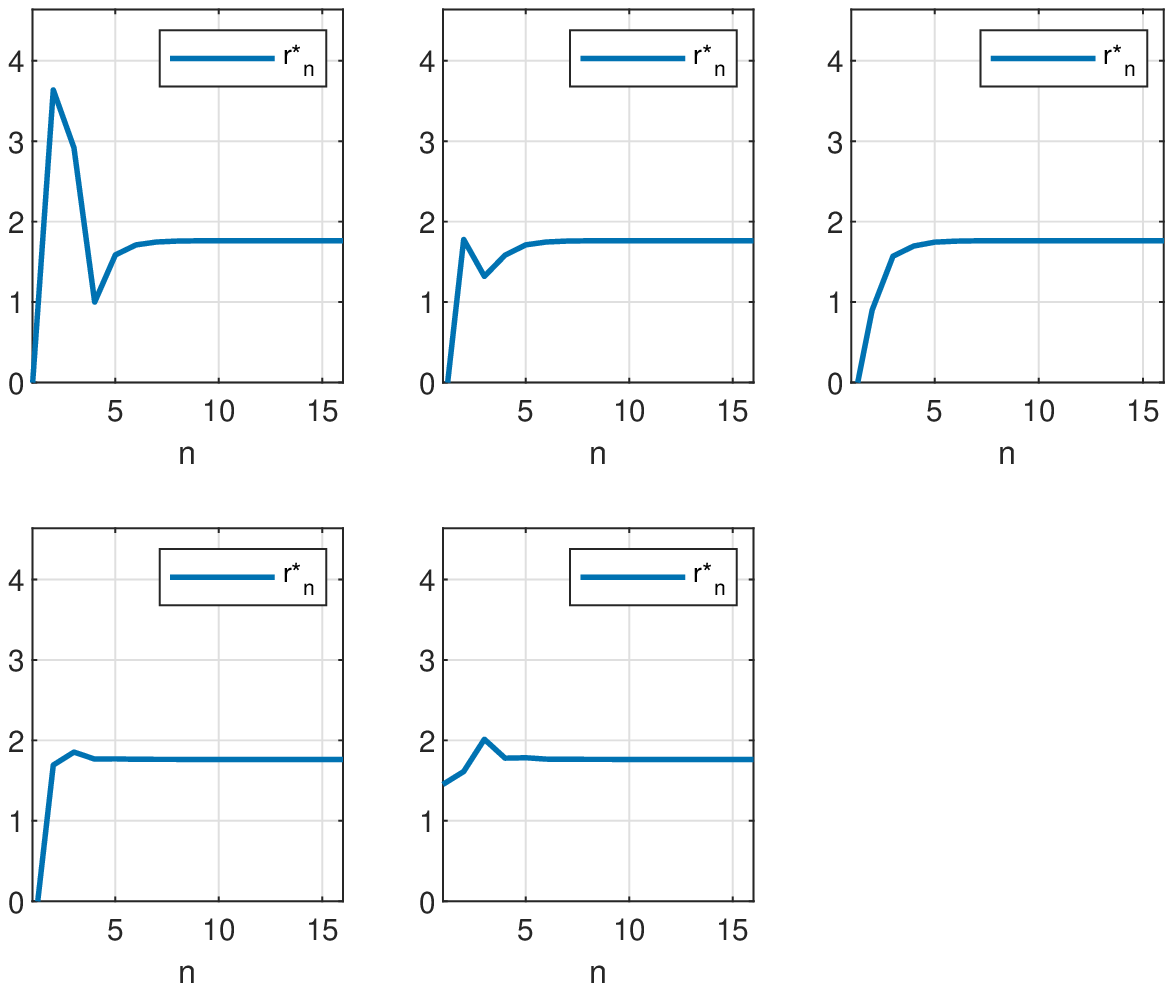}
\end{center}
\end{minipage}

\caption{Semi-regular interpolatory scheme based on the polyharmonic function $g(x)=|x|$, $L=2$ and $m=3$ on the mesh $\bf{t}$  with $h_\ell=1$, $h_r=2$ analyzed with the semi-regular Dubuc-Deslauriers $6$-point scheme.  Top row: part of the subdivision matrix $\mathbf{Z}$ that corresponds to the non-shift-invariant refinable functions around $0$ and the graphs of these functions $\zeta_{-2},\dots,\zeta_{2}$. Middle row: estimates of the optimal H\"older-Zygmund exponents of $\zeta_{-2},\dots,\zeta_{2}$ via linear regression (on the left) via the method in Proposition \ref{prop:my_test} (on the right).
Bottom row: graphs of the estimates of the H\"older-Zygmund exponents.}
\label{fig:PS_DD6}
\end{figure}

Similar results, enlightening the viability of our method,  appear for a large number of other members of the semi-regular families of 
subdivision schemes (i.e. B-splines, Dubuc-Deslauriers and interpolatory schemes based on (inverse) multi-quadrics, gaussians, Wendland's functions, Wu's functions, Buhmann's functions, polyharmonic functions and Euclid's hat functions \cite{MR2357267}). For the interested reader, a MATLAB function for the generation of semi-regular RBFs-based interpolatory schemes is available at \cite{RBFschemes}.


$ $\\
{\bf Acknowledgements:} The authors  thank Karlheinz Gr\"ochenig for his fruitful suggestions
and the Erwin Schr\"odinger International Institute for Mathematics and Physics (ESI), Vienna, Austria, for
discussion stimulating environment. Maria Charina
is sponsored by the Austrian Science Foundation (FWF) grant P28287-N35. Costanza Conti, Lucia Romani and Alberto Viscardi
have conducted this research within Research ITalian network on Approximation (RITA).

\section*{References}

\bibliographystyle{siam}
\bibliography{charina_conti_romani_stoeckler_viscardi}

\end{document}